\documentclass[11pt, final]{article}

\usepackage{a4}
\usepackage{amsmath}%
\usepackage{amstext}%
\usepackage{amssymb}%
\usepackage[utf8]{inputenc}
\usepackage{tikz}
\usetikzlibrary{arrows.meta,shapes,positioning,shadows,trees}
\usepackage{amsxtra}
\usepackage{graphicx}
\usepackage{showkeys}%
\usepackage{epsfig}%
\usepackage{cite}
\usepackage{caption}
\usepackage{longtable}
\usepackage{threeparttable,booktabs,multirow,lscape}
\usepackage{lipsum}
\usepackage{rotating}
\usepackage{multirow}
\usepackage{booktabs}
\usepackage{floatrow}
\usepackage{amsmath}
\usepackage{amssymb}
\usepackage{graphicx}
\usepackage[colorlinks,urlcolor=red]{hyperref}
\usepackage{graphicx}
\usepackage{subfigure}
\usepackage{listings}
\usepackage[ruled,vlined]{algorithm2e}

\newtheorem{theorem}{Theorem}

\newtheorem{assumption}[theorem]{Assumption}

\newtheorem{lemma}[theorem]{Lemma}

\newtheorem{proposition}[theorem]{Proposition}
\newtheorem{rem}[theorem]{Remark}
\newenvironment{remark}{\rem\rm}{\endrem}

\newcounter{unnumber}

\newtheorem{prf}[unnumber]{Proof.}
\newenvironment{proof}{\prf\rm}{\hfill{$\blacksquare$}\endprf}
\newcommand{\R}{\mathbb{R}}%
\newcommand{\N}{\mathbb{N}}%

\renewcommand{\>}{\right\rangle}
\newcommand{\<}{\left\langle}

\DeclareMathOperator*\Id{Id}%

\DeclareMathOperator*\fix{Fix}
\DeclareMathOperator*\im{Im}

\title{Subgradient Splitting Methods for Nonsmooth Fractional Programming with  Fixed-Point Constraints}

\author{ 
Mootta Prangprakhon\thanks{Department of Mathematics, Faculty of Science, Khon Kaen University, Khon Kaen 40002, Thailand,
	email: mootta\_prangprakhon@hotmail.com.}
	\and
	Nimit  Nimana\thanks{Department of Mathematics, Faculty of Science, Khon Kaen University, Khon Kaen, 40002 Thailand,
		email: nimitni@kku.ac.th.}}
	
\begin{document}
		\maketitle
	\begin{abstract}
 We consider a class of nonsmooth fractional programming problems with fixed-point constraints, where the numerator is convex and the denominator is concave.
To solve this problem, we propose splitting algorithms that compute subgradient steps separately for the convex numerator and the concave denominator. These methods offer a straightforward approach by eliminating the need to solve subproblems at each iteration. By leveraging fixed-point constraints, the proposed algorithms are particularly well-suited for problems with complex constraint structures.
Under certain assumptions, we establish the convergence of the proposed methods.
Furthermore, to address large-scale optimization, we propose an incremental subgradient algorithm for a class of nonsmooth sum-of-ratios fractional programming problems and analyze its convergence. 
Finally, we present numerical experiments, including comparative analyses of our algorithms with existing methods, to demonstrate the effectiveness and performance of the proposed approach.
	
		
	\end{abstract}

		
		
		\section{Introduction}
Many problems across various domains of mathematics and management science involve the minimization or maximization of an objective function that appears as one or several ratios of functions 
\cite{SY18-1,CZS23,DKM24}.
This scenario is commonly referred to as {\it fractional programming.}
Due to its ability to address a wide range of real-world problems, fractional programming has gained increasing popularity and prominence in recent years.
Consequently, there is now a considerable amount of literature dedicated to exploring the various applications of fractional programming across diverse fields, including economics and finance
\cite{LB16,SM17}, 
engineering \cite{ZJ15,ZBSJ17,CYH13}, 
and operations research \cite{ADS17,DJF16}.
For a more comprehensive overview and applications on fractional programming, we refer to the works of Stancu-Minasia \cite{SM97,SM19} and the references therein.

In this work, we focus on solving a class of nonsmooth fractional programming problems subject to a fixed-point constraint:
 \begin{equation} \label{main-FP}\displaystyle
 \tag{P}
	\begin{array}{ll}
    \mathrm{minimize}\hskip0.2cm\indent  \displaystyle \frac{f(x)}{g(x)}\\
	\mathrm{subject \hskip0.2cm to}\indent \displaystyle  x\in\fix T,
	\end{array}  
 \end{equation}
where the following assumptions hold:
\begin{itemize}
\item[(A1)] $T:\R^k\to\R^k$ is a $\rho$-strongly quasi-nonexpansive operator 
with $\rho>0$;
\item[(A2)]  $f:\R^k\to\R$ is a convex function and $f(x)\geq0$ for all $x\in\im T$;
\item[(A3)] $g:\R^k\to\R$ is a concave function and there exists $M>0$ such that $0<g(x)\leq M$ for all $x\in\im T$,
\end{itemize}
with $\R^k$ denoting a Euclidean space. We define $\theta^*:=\min\limits_{x\in \fix T}\frac{f(x)}{g(x)}$, and denote the solution set of Problem (\ref{main-FP}) as $\mathcal{S}:=\left\{x^*\in \fix T:\frac{f(x^*)}{g(x^*)}=\theta^*\right\},$
providing that the solution set $\mathcal{S}$ is nonempty.


\subsection{Existing Methods for Fractional Programming}

In the literature, several approaches and algorithms have been proposed to handle fractional programming. A classical method for solving such problems is the parametric approach \cite{J66,D67}, which involves reformulating the original problem into a parametric one. According to this approach, a solution to the fractional programming problem can be obtained by solving the corresponding parametric problem. To address this reformulation, {\it Dinkelbach’s algorithm} (DA) \cite{D67} was introduced. This algorithm requires solving a subproblem at each iteration, where both the numerator and denominator functions must be considered simultaneously. The convergence of this method has been established. Moreover, the algorithm has been generalized and further studied by several authors, such as Schaible \cite{S76}, Ibaraki \cite{I83}, and Crouzeix et al. \cite{CFS85}.
 However, it is worth noting that solving the subproblem in each iteration may involve high computational cost and can be as challenging as solving the original fractional programming problem.

While the parametric approach is widely known in fractional programming, there are other classical approaches such as the variable transformation approach \cite{CC62}, which involves transforming the original problem into a more solvable problem, and existing nonlinear programming methods \cite{HI20,B77,KE67}, which utilize certain properties of the entire objective function to find the optimal solution. 
For more details on this discussion, we refer readers to the works of Schaible \cite{S81} and Stancu-Minasian \cite{SM97}.

To propose an alternative direct approach,
Bo\c{t} and Csetnek \cite{BotC17} presented two splitting algorithms called {\it proximal-gradient algorithms} (PGA), which are applicable when the numerator is convex and the denominator is either smooth concave or smooth convex. A key advantage of this approach is that it allows the two functions to be handled separately; specifically, the numerator performs the proximal step while the denominator performs the gradient step at each iteration.
More recent studies have extended this idea by developing splitting-based algorithms with various modifications and acceleration techniques to address specific cases of fractional programming, particularly when the denominator is nonsmooth and convex. These methods typically incorporate the use of subgradients at each iteration to manage the nonsmoothness; see, for example, Zhang and Li \cite{ZL22}, Bo\c{t} et al. \cite{BDL22}, and Li et al. \cite{LSZZ22}. 
In addition, Prangprakhon et al.  \cite{PFN22}   proposed 
a projection-based splitting algorithm called an {\it adaptive projection gradient method} (APGM) for solving smooth convex-concave fractional programming problems. It is worth mentioning that a notable aspect of this work is the nonincreasing property of the adaptive step-size sequence, which can be guaranteed in the paper.

\subsection{Incremental Schemes for Finite-Sum Structure} 

When dealing with large-scale convex optimization tasks where the number of component functions
is large, it is of  interest to consider methods that operate on each component function at each iteration, rather than on the entire objective function. These methods are referred to as {\it incremental methods}. Such methods encompass a variety of approaches that utilize different types of tools, such as gradient, subgradient, and proximal operator.
Among these, we specifically focus on {\it incremental subgradient methods} (ISMs), which are designed to handle a situation where each component function is nondifferentiable. 

ISMs have a long history in optimization. For more extensive details and convergence analyses, we refer the readers to the works of Nedi\'{c} and Bertsekas \cite{NB01}, and Bertsekas \cite{B10}, as well as the references included therein. 
Additionally, within the context of solving fractional programming problems, it is noteworthy that very recently Bo\c{t} et al. \cite{BDL23} have proposed an {\it inertial proximal subgradient method} that relates to incremental-type methods to solve a class of nonsmooth sum-of-ratios fractional programming problems. 

\subsection{Convex Optimization with Fixed-Point Constraints}

Note that projection-based algorithms, such as the APGM, use the metric projection as their main operator, and the ability to compute it depends on the simplicity of the constraint set. Specifically, the constraint set must be simple enough so that a closed-form expression of the metric projection exists. However, in practical scenarios, the constraint set can be highly complex. For example, it may be defined as the intersection of multiple closed and convex subsets, or as the solution set of a nonlinear problem. In such cases, computing the metric projection becomes challenging. 

To address this type of obstacle, previous literature by Yamada \cite{Y01} 
introduced the idea of using a nonexpansive operator 
instead of the metric projection, and interpreting the constraint set of the problem as the fixed-point set.
It is well-known that the fixed-point sets of nonexpansive operators have important properties. For instance, 
one can construct a nonexpansive operator whose fixed-point set coincides with the intersection of the fixed-point sets of other nonexpansive operators. This property is particularly useful when dealing with problems involving multiple constraints.

Motivated by these, Yamada proposed the {\it hybrid steepest descent method} (HSDM), specifically developed to solve convex minimization problems where the constraint is the fixed point set of the nonexpansive operator. Note that in Yamada's work \cite{Y01}, the transformation occurs not only in the algorithm itself but also in the formulation of the constraint set. This shift to fixed-point constraints in optimization has also been applied to problems where the objective function is a sum of convex component functions. Such problems, where the constraint sets are fixed-point sets, allow for the exploration of constrained optimization scenarios in which the closed-form expression of the metric projection onto the constraint set may not be known. 
Problems of this nature have been discussed in various works \cite{IH14,IP16,I18} and cover a range of domains, including optimal control \cite{I12}, bandwidth allocation \cite{I13}, and ensemble learning \cite{I20}. Moreover, it is worth mentioning that some of these works employ operators from broader classes than nonexpansive operators to address more general problems.

Due to similar concerns with the finite-sum structure, Iiduka \cite{I16} proposed an {\it incremental subgradient method} (I-ISM) for solve minimization problems involving a sum of component functions, where the constraint is the intersection of fixed-point sets. The algorithm builds on the conventional incremental approach, which employs the metric projection as its main operator. 

 To the best of our knowledge, the direct approach for solving both single-ratio fractional programming problems and sum-of-ratios fractional programming problems, which involve fixed-point constraints, has seldom been studied in prior literature.
 This is of particular interest and plays a key role in the development of the presented work. 


\subsection{Contributions of the Paper}

With the aforementioned literature overviews and motivations in mind, the contributions of this work are as follows:

\begin{enumerate}
    \item \textbf{Fixed-Point Subgradient Splitting Method (FSSM):} 
    We propose the FSSM to solve a class of nonsmooth convex-concave fractional programming problems with fixed-point constraints. The method computes subgradients of the numerator and denominator separately at each step. Under certain conditions on the step size and the boundedness of the sequence $\{x_n\}_{n=1}^{\infty}$, we establish  subsequential convergence.
    
    \item \textbf{Adaptive Fixed-Point Subgradient Splitting Method (AFSSM):}
    In an attempt to remove the boundedness assumption of the sequence $\{x^n\} _{n = 1}^\infty$ in the FSSM, we develop the AFSSM, a modified version of the FSSM.
    Under appropriate  conditions and the strong convexity of the function $f$, we obtain strong convergence.

    \item \textbf{Incremental Fixed-Point Subgradient Splitting Method (IFSSM):} 
    To tackle large-scale optimization problems, we propose the IFSSM,  designed to solve a class of sum-of-ratios fractional programming problems over the intersection of fixed-point sets. 
   This algorithm is inspired by the conventional ISM and Iiduka's I-ISM \cite{I16}.  Under appropriate assumptions, we obtain subsequential convergence.

  \item \textbf{Key Advantages:} 
        The three methods proposed in this work share two key advantages. First, they provide an alternative to existing techniques, such as the DA \cite{D67} and methods based on proximal operators, by eliminating the need to solve subproblems at each iteration. This simplification makes the optimization process more straightforward and has the potential to reduce both computational time and effort.
Second, by working on fixed-point constraints, the proposed methods are  well-suited for handling problems with complex constraint structures.

    \item \textbf{Numerical Examples:} 
To demonstrate the effectiveness and performance of the proposed methods, we apply our algorithms to solve three types of fractional programming problems and compare their performance with existing methods. The problems are: (i) a smooth fractional programming problem with linear constraints; (ii) a cost-to-profit minimization problem where the profit is modeled by a Cobb–Douglas production function; and (iii) the minimization of a sum of ratios of linear functions.
 \end{enumerate}

\subsection{Organization}

This paper is organized as follows: 
In Section \ref{sect-2}, we introduce the notation and recall important definitions, propositions, and lemmas.
Section \ref{algorithm 12} is divided into two parts: Subsection \ref{sect-3.1} presents the FSSM 
and its general convergence analysis, while Subsection \ref{sect-3.2} presents the AFSSM, also with its convergence analysis.
 Section \ref{sect-incremental} is dedicated to the presentation of the IFSSM and its convergence analysis. 
 Finally, in Section \ref{sect-numerical}, we illustrate the performance and effectiveness of the proposed methods through three numerical examples.

\section{Preliminaries}\label{sect-2}
	
		In this section, we introduce fundamental concepts and properties that will be used in the following sections. 
	Throughout this paper, we assume that $\R^k$ is a Euclidean space with the inner product $\langle \cdot, \cdot \rangle$ and its induced norm $\|\cdot\|$. We denote by $\Id$ the identity operator on $\R^k$. The  strong convergence of $\{ {x^n}\} _{n =1}^\infty$ to $x\in\R^k$ is represented by ${x^n} \to x$. For an operator $T:\R^k\to\R^k$, $\fix T:=\{x\in\R^k:Tx=x\}$ and $\im T:=\{y\in\R^k:y=Tx,\exists x\in\R^k\}$ denote the set of fixed points and  images of $T$, respectively. 

We recall some nonlinear operators for later use; further details can be found in Cegielski’s book \cite[Chapters 2 and 3]{C12}. Let $C\subset\R^k$ and $x\in\R^k$ be given. If there exists a point $y\in C$ such that $\|y-x\|\le\|z-x\|$ for all $z\in C$, then $y$ is called the {\it metric projection} of $x$ onto $C$, denoted by $P_Cx$. If $C$ is a nonempty closed and convex subset, then for any $x\in\R^k$, $P_Cx$ exists and is uniquely defined (see, e.g., \cite[Theorem 1.2.3]{C12}).
    For an operator $T:\R^k\to\R^k$ 
  with $\fix T\ne\emptyset$, 
  $T$ is said to be {\it quasi-nonexpansive} (QNE) if 
	$\|Tx-z\|\le\|x-z\|$ for all $x\in\R^k$ and  $z\in\fix T$; {\it $\rho$-strongly quasi-nonexpansive} ($\rho$-SQNE), $\rho\ge0$, if 
	$\|Tx-z\|^2\le\|x-z\|^2-\rho\|Tx-x\|^2$; a {\it cutter} if 
$\langle z-Tx,x-Tx\rangle\le 0$.
Moreover, an operator $T:\R^k\to\R^k$ is said to be {\it nonexpansive} (NE) if 
 $\|Tx-Ty\|\le\|x-y\|$ for all $x,y\in\R^k$; {\it $\rho$-firmly nonexpansive} ($\rho$-FNE) \cite[Definition 2.1]{C06}, where $\rho\ge0$, if 
 $\|Tx-Ty\|^2\le\|x-y\|^2 - \rho\|(Tx-x)-(Ty-y)\|^2$; {\it firmly nonexpansive} (FNE) if 
 $\langle Tx-Ty,x-y\rangle\ge\|Tx-Ty\|^2$. 
 Clearly, every $\rho$-SQNE operator is QNE, and every $\rho$-FNE operator is $\rho$-SQNE. Furthermore, $T$ is a cutter if and only if it is 1-SQNE \cite[Theorem 2.1.39]{C12}, and since FNE implies 1-FNE, every FNE operator is a cutter and thus also NE. Notably, the metric projection is FNE, a cutter, and NE. For further details on these relationships, see Reich and Zalas \cite{RZ16}.
Now, we recall the demi-closedness principle.
 An operator $T:\R^k \to \R^k$ satisfies the {\it demi-closedness principle} if $T-\Id$ is demi-closed at $0$; that it, for any sequence $\{x^n\}_{n=1}^\infty\subset\R^k$, if $x^n\to x\in\R^k$ and $\|(T-\Id)x^n\| \to 0$, then $x\in\fix T$.
It is worth mentioning that if $T:\R^k\to\R^k$ is a nonexpansive operator with $\fix T\neq\emptyset$, then  $T$ satisfies the demi-closedness principle; see \cite[Lemma 3.2.5]{C12}.



Next, we recall the concepts of convexity, strong convexity, and subgradients. For more comprehensive details, we refer the reader to the books by Bauschke and Combettes \cite{BC17}, and Beck \cite{B17}.
A function $f:\R^k\to\R$  is said to be {\it convex} 
if 
$f(\alpha x+ (1-\alpha)y)\le \alpha f(x)+(1-\alpha)f(y)$ for all $x,y\in\R^k$ and $\alpha\in[0,1]$;
{\it $\rho$-strongly convex}, where $\rho>0$, if 
$f(\alpha x+ (1-\alpha)y)\le \alpha f(x)+(1-\alpha)f(y)-\frac{\rho}{2}\alpha(1-\alpha)\|x-y\|^2$; and 
{\it concave} if $-f$ is convex.
Furthermore, let $f:\R^k\to\R$ be a convex function. A vector $f'(x)\in\R^k$ is called a {\it subgradient} of $f$ at $x\in\R^k$ if the {\it subgradient inequality}  
$\langle f'(x),y-x\rangle \le f(y)-f(x)$ holds for all $y\in\R^k$.
Moreover, the set of all subgradients of  $f$ at $x$ is called the {\it subdifferential} of $f$ at $x$, and is defined as 
${\partial} f(x):=\{f'(x)\in\R^k: \langle f'(x),y-x\rangle \le f(y)-f(x) \quad  \text{for all $y\in\R^k$}\}.$

The following results are crucial to the convergence analysis.

 \begin{proposition} (\cite{BC17}, Proposition 16.20).\label{fact-BC17}
 Let $f:\R^k\to\R$ be a convex function. Then the following conditions hold:
\begin{itemize}
    \item[(i)]  $f$ is continuous.
    \item[(ii)]  $f$  is bounded on every bounded subset of $\R^k$. 
    \item[(iii)] $\partial f$ maps every bounded subset of $\R^k$ to a bounded set.  
\end{itemize}
\end{proposition}

\begin{lemma}
(\cite{B17}, Theorem 5.24).
\label{lemma-strongconvex}
    Let $f:\R^k\to\R$ be a $\sigma$-strongly convex function with $\sigma>0$.
    Then,
     the following inequality holds:
    \begin{equation*}
      \sigma\|x-y\|^2 \le \langle f'(x) - f'(y), x-y\rangle,
    \end{equation*}
    for all $x,y\in\R^k$, where $f'(x)\in\partial f(x)$ and $f'(y)\in\partial f(y)$.
\end{lemma}


\begin{lemma}(\cite{M08}, Lemma 3.1).\label{paul}
		Let $\{ {a_n}\} _{n = 1}^\infty$ 
		be a nonnegative sequence such that there exists a subsequence $\{ {a_{n_j}}\} _{j = 1}^\infty$ of $\{ {a_n}\} _{n = 1}^\infty$ with $a_{n_j}< a_{n_{j+1}}$ for all $j\ge1$. Define, for all $n\ge n_0$, 
		${\nu(n)} = \max\left\{ {k\in\mathbb{N}: n_0\le k \le n, {a_k} < {a_{k + 1}}} \right\}$.
		 Then, $\{\nu(n)\}_{n\ge n_0}$ is a nondecreasing sequence such that $\mathop {\lim }\limits_{n \to \infty } {\nu(n)} = \infty$ and, for all $n\ge n_0$,
	$a_{\nu(n)}\le a_{\nu(n)+1}$ and $a_{n}\le a_{\nu(n)+1}$.
\end{lemma}


  \section{Fixed-Point Subgradient Splitting Method}\label{algorithm 12}

In this section, we present an algorithm to solve Problem (\ref{main-FP}). Before doing so, we first introduce the necessary notations. Let $f^\prime(x)$ and $g^\prime(x)$ denote the subgradients of $f$ and $g$ at $x\in\R^k$, respectively. For convenience, we define $h^\prime(x):= (-g)^\prime(x)$, which represents the subgradient of $-g$ at $x$.
We now present the following method to solve Problem (\ref{main-FP}):
	
		\vskip2mm
		\begin{algorithm}[H]
			\SetAlgoLined
			\vskip2mm
			\textbf{Initialization}: Choose a stepsize sequence $\{\eta_n\}_{n=1}^\infty\subset(0,\infty)$ and take an initial point $x^1\in\R^k$  with $g(x_1)>0$. Set $\theta_1=\frac{f(x^1)}{g(x^1)}$.

\textbf{Iterative step}: For a current iterate $x^n\in \im T$ ($n\in\N$), calculate as follows:

\vspace{0.3cm}

\textbf{Step 1.} Compute $\theta_n=\frac{f(x^n)}{g(x^n)}.$ 
\vspace{0.3cm}

\textbf{Step 2.} Compute subgradients $f^\prime (x^n)\in\partial f(x^n)$ and $h^\prime (x^n)\in\partial (-g)(x^n)$ and calculate the next iterate $x^{n+1}\in \im T$ as 
			\begin{equation*}
    x^{n+1}=T(x^n-\eta_n f^\prime (x^n) - \eta_n\theta_n h^\prime (x^n)).
\end{equation*}
Update $n:=n+1$ and go to \textbf{Step 1.}
			\caption{Fixed-Point Subgradient Splitting Method (FSSM)}
			\label{main-algor}
			\vskip2mm
			
		\end{algorithm}
		
		\vskip2mm
		

Before proceeding with the convergence analysis, we make some significant remarks.

  	\begin{remark}  Some significant remarks are in order:
  	\begin{itemize}

    \item[(i)]  One challenging aspect of the FSSM,  which sets it apart from previous works beyond the use of the main operator, is that we cannot guarantee that  $\theta^*\le\theta_n$ for all $n\in\N$. This is due to the fact that in the FSSM, $x^*\in\fix T$ while the generated sequences $x^n$ need not belong in the constraint $\fix T$.
   Therefore, $\theta^*$ and $\theta_n$  are incomparable in this context.
  	    \item[(ii)] One can observe that in a special case where $T:=P_C$, the metric projection onto a nonempty closed and convex subset, the FSSM is related to the APGM introduced very recently in Prangprakhon et al. \cite{PFN22}. 
  	\end{itemize}
  		
  	\end{remark}


\subsection{General Convergence}\label{sect-3.1}

In this section, we study the general convergence of the FSSM. Let $\{x^n\} _{n = 1}^\infty$ be a sequence generated by the FSSM, and for simplicity, define
\begin{equation*}
    y^n:=x^n-\eta_n f^\prime (x^n) - \eta_n\theta_n h^\prime (x^n), \quad \text{for all $n\in\N$.}
\end{equation*}

We begin with the following fundamental lemma, which is useful in the convergence analysis.
     \begin{lemma}\label{lm-1-bdd} 
    Let $\{x^n\} _{n = 1}^\infty$ be a sequence generated by the FSSM.  If  $\{x^n\}_{n=1}^\infty$ is bounded, then for any $z\in \fix T$, the following statement hold:
 \begin{equation*}
     \|x^{n+1}-z\|^2 \le \|x^n-z\|^2 + 2g(z)\eta_n\left(\frac{f(z)}{g(z)}-\theta_n\right) + B\eta_n^2 - \rho\|Ty^n - y^n\|^2,
\end{equation*} 
    where $B:= \sup\limits_{n\ge 1}\|f^\prime (x^n) +\theta_n h^\prime (x^n)\|^2$.		
    \end{lemma}

   \begin{proof} 
Let $z\in\fix T$ and suppose that  $\{x^n\} _{n = 1}^\infty$ is bounded. 
We set $B:= \sup\limits_{n\ge 1}\|f^\prime (x^n) +\theta_n h^\prime (x^n)\|^2$.
From the $\rho$-strong quasi-nonexpansivity of $T$, we obtain
\begin{eqnarray*}\label{eq1}
    \|x^{n+1}-z\|^2
    &\le& \|x^n-\eta_n f^\prime (x^n) - \eta_n\theta_n h^\prime (x^n)-z\|^2 - \rho\|Ty^n -y^n\|^2\nonumber\\
    &=&\|x^n-z\|^2 + \eta_n^2\|f^\prime(x^n)+\theta_n h^\prime(x^n)\|^2 \nonumber\\
    && - 2\eta_n\langle x^n-z, f^\prime(x^n)+\theta_n h^\prime(x^n)\rangle - \rho\|Ty^n -y^n\|^2\nonumber\\
&\le&\|x^n-z\|^2 - 2\eta_n\langle x^n-z, f^\prime(x^n)\rangle - 2\eta_n\theta_n\langle x^n-z, h^\prime(x^n)\rangle \nonumber\\
    &&+ B\eta_n^2  - \rho\|Ty^n -y^n\|^2.
     \end{eqnarray*}
Since $g$ is concave, $-g$ is convex. The subgradient inequalities of $f$ and $-g$ yield that
\begin{eqnarray*}\label{eq2}
    \|x^{n+1}-z\|
    &\le& \|x^n-z\|^2 + 2\eta_n(f(z)-\theta_ng(z)) + 2\eta_n(\theta_ng(x^n)-f(x^n)) \nonumber\\
    && +B\eta_n^2  - \rho\|Ty^n -y^n\|^2\nonumber\\
    &=&\|x^n-z\|^2 + 2\eta_ng(z)\left(\frac{f(z)}{g(z)}-\theta_n\right) + B\eta_n^2  - \rho\|Ty^n -y^n\|^2,
\end{eqnarray*}
which completes the proof.
   \end{proof}

We first show the convergence behavior of the sequence of objective function values as follows.
    \begin{theorem}\label{thm-1-obj} 
    Let $\{x^n\} _{n = 1}^\infty$ be a sequence generated by the FSSM.  If  $\{x^n\}_{n=1}^\infty$ is bounded, then for any $z\in \fix T$, the following statements hold:
\begin{itemize}
    \item[(i)] If $\eta_n=\eta_0$, for all $n\in\N$, where $\eta_0>0$, then, for each $K\in\N$, the following inequality holds:
\begin{equation*}
\min\limits_{1\le n \le K}\theta_n - \frac{f(z)}{g(z)} \le \frac{\|x^1-z\|^2}{2g(z)\eta_0 K}+\frac{\eta_0B}{2g(z)},
\end{equation*}
    where $B$ is given in Lemma \ref{lm-1-bdd}.
   \item[(ii)] If  $\sum\limits_{n=1}^{\infty}\eta_n=\infty$ and $\sum\limits_{n=1}^{\infty}\eta_n^2 < \infty$, then
$$\liminf\limits_{n\rightarrow\infty}\theta_n \le \frac{f(z)}{g(z)}.$$ 
\item[(iii)] If $\eta_n=\frac{1}{(n+1)^p}$, for all $n\in\N$, where $p>0.5$, then, for each $K\in\N$, the following inequality holds:
\begin{equation*}
\min\limits_{1\le n \le K}\theta_n - \frac{f(z)}{g(z)} \le 
 \mathcal{O}\left( \frac{1}{K^{1-p}} \right).
\end{equation*}
\end{itemize}		
    \end{theorem}

   \begin{proof} 
Let $z\in\fix T$ and suppose that  $\{x^n\} _{n = 1}^\infty$ is bounded. 
First, by omitting the nonpositive term $\rho\|Ty^n -y^n\|^2$ and summing inequality in Lemma \ref{lm-1-bdd} from $n=1$ to  a fixed $K$, we have that
    \begin{eqnarray}\label{eqn-i}
  	2g(z)\sum\limits_{n = 1}^K\eta_n\left(\theta_n - \frac{f(z)}{g(z)}\right)
&\le& \|x^{1}-z\|^2 +B\sum\limits_{n = 1}^K \eta_n^2. 
   \end{eqnarray}

(i) Let $\eta_0>0$ be given. By putting $\eta_n=\eta_0$ in (\ref{eqn-i}), we have 
 \begin{eqnarray*}
  	2g(z)K\eta_0\min\limits_{1\leq n\leq K}\left(\theta_n - \frac{f(z)}{g(z)}\right)
&\le& \|x^{1}-z\|^2 +\eta_0^2BK, 
   \end{eqnarray*}
   which proves (i).

    We next prove (ii). By approaching $K$ to infinity in (\ref{eqn-i}), we obtain
    \begin{eqnarray*}
  	2g(z)\sum\limits_{n = 1}^\infty\eta_n\left(\theta_n - \frac{f(z)}{g(z)}\right)
&\le& \|x^{1}-z\|^2 +B\sum\limits_{n = 1}^\infty \eta_n^2, 
   \end{eqnarray*}
   which implies that 
    $\sum\limits_{n =1}^\infty\eta_n\left(\theta_n - \frac{f(z)}{g(z)}\right)< \infty.$
   Next, we shall show that $\liminf\limits_{n\rightarrow\infty}\theta_n \le \frac{f(z)}{g(z)}$ by contradiction. Assume that there is an $\varepsilon>0$ and $N\in\N$ such that $\theta_n-\frac{f(z)}{g(z)}\ge\varepsilon$ for all $n\ge N$. Since $\{\eta_n\}_{n=1}^\infty\subset(0,\infty)$ and $\sum\limits_{n=1}^{\infty}\eta_n=\infty$, we thus have
   \begin{equation*}
       \infty = \varepsilon\sum\limits_{n = 1}^\infty\eta_n \le \sum\limits_{n =1}^\infty\eta_n\left(\theta_n - \frac{f(z)}{g(z)}\right)< \infty,
   \end{equation*}
 which brings a contradiction. Then, the result of (ii) follows 
 immediately. 
 
 Next we prove (iii). By putting $\eta_n=\frac{1}{(n+1)^p}$ in (\ref{eqn-i}), we have 
\begin{equation}\label{remark-rate-1}
\min\limits_{1\le n \le K}\theta_n - \frac{f(z)}{g(z)} \le \frac{\frac{1}{2g(z)}\|x^{1}-z\|^2 + \frac{B}{2g(z)}\sum\limits_{n = 1}^K \frac{1}{(n+1)^{2p}}}{\sum\limits_{n = 1}^K\frac{1}{(n+1)^p}}.
\end{equation}
Due to \cite[Lemma 8.27]{B17}, we obtain the following inequalities: 
\begin{equation*}
  \sum\limits_{n = 1}^K \frac{1}{(n+1)^{2p}}\le\int_0^K \frac{1}{(t+1)^{2p}} \,dt = \frac{(K+1)^{1-2p}-1}{1-2p} \le \frac{1}{2p - 1},
  \end{equation*}
  and
    \begin{equation*}
  \sum\limits_{n = 1}^K \frac{1}{(n+1)^{p}}\ge\int_1^{K+1} \frac{1}{(t+1)^p} \,dt = \frac{(K+2)^{1-p} - 2^{1-p}}{1-p}, 
  \end{equation*}
  for all $K\in\N$. Furthermore, let us note further that 
\begin{equation*}
     \frac{(K+2)^{1-p} - 2^{1-p}}{1-p} = 
     \frac{\left[ \frac{(K+2)^{1-p} - 2^{1-p}}{(K+3)^{1-p}} \right] (K+3)^{1-p}}{1-p}
\ge \frac{\left( \frac{3}{4} \right)^{1-p} - \left( \frac{1}{2} \right)^{1-p}}{(1-p)(K+3)^{p-1}}.
\end{equation*}
Hence, the relation (\ref{remark-rate-1}) becomes 
\begin{equation*}
\min\limits_{1\le n \le K}\theta_n - \frac{f(z)}{g(z)} \le 
\left[ \frac{\frac{(1-p)}{2g(z)}\|x^{1}-z\|^2 + \frac{B(1-p)}{(2g(z))(2p - 1)}}{\left( \frac{3}{4} \right)^{1-p} - \left( \frac{1}{2} \right)^{1-p}} \right] \cdot\ (K+3)^{p-1} \le \mathcal{O}\left( \frac{1}{K^{1-p}} \right).
\end{equation*} The proof is complete.  
   \end{proof}

		
	The following is the convergence theorem of the FSSM in the setting of Problem (\ref{main-FP}).
		
	
	    \begin{theorem}\label{main-thm-bdd}
			Let $\{x^n\} _{n = 1}^\infty$ be a sequence generated by the FSSM. Suppose that $\sum\limits_{n=1}^{\infty}\eta_n=\infty$ and $\sum\limits_{n=1}^{\infty}\eta_n^2 < \infty$.
			 If $\{x^n\} _{n = 1}^\infty$ is bounded and
             the operator $T$ satisfies the demi-closedness principle, then there exists a subsequence of the sequence $\{x^n\}_{n=1}^\infty$ that converges to a point in $\mathcal{S}$.
		\end{theorem}
		
		\begin{proof}
			Let $z\in\fix T$.
				Suppose that  $\{x^n\} _{n = 1}^\infty$ is bounded and that $T$ satisfies the demi-closedness principle. 
 			To prove the theorem’s convergence, we consider two cases based on the behavior of the sequence $\{\|x^n-z\|^2\}_{n=1}^\infty$:
			
	\indent\textbf{Case 1.} Assume that there is an ${n_0}\in\mathbb{N}$ such that $\{\|x^n-z\|^2\}_{n\geq n_0}$ is nonincreasing for all $z\in \fix T$. 
    Then, $\mathop {\lim }\limits_{n \to \infty }{\|x^n-z\|^2}$ exists. 
   Define $B_1:=\sup\limits_{n\ge 1}\left|\frac{f(z)}{g(z)} - \theta_n\right|$. By Lemma \ref{lm-1-bdd}, we obtain  
\begin{equation*}
   \rho\limsup\limits_{n\rightarrow\infty}\|Ty^n - y^n\|
  \le\lim\limits_{n\rightarrow\infty}\|x^{n}-z\|^2 - \lim\limits_{n\rightarrow\infty}\|x^{n+1}-z\|^2 + 2g(z)B_1\lim\limits_{n\rightarrow\infty}\eta_n + B\lim\limits_{n\rightarrow\infty}\eta_n^2.
\end{equation*}
Since $\sum\limits_{n=1}^{\infty}\eta_n^2 < \infty$, $\rho>0$ and $g(z)>0$, we deduce that 
\begin{equation}\label{eq5}
\lim\limits_{n\rightarrow\infty}\|Ty^n-y^n\|=0.
\end{equation}
Due to the boundedness of $\{x^n\}_{n=1}^\infty$ and Theorem \ref{thm-1-obj}(ii), there exists a subsequence $\{x^{n_k}\}_{k=1}^\infty$ of $\{x^n\}_{n=1}^\infty$ such that 
   \begin{equation}\label{eq6}
\lim\limits_{k\rightarrow\infty} \frac{f(x^{n_k})}{g(x^{n_k})} = \liminf\limits_{n\rightarrow\infty}\frac{f(x^n)}{g(x^n)} \le \frac{f(z)}{g(z)},
\end{equation}
for all $z\in \fix T$. Additionally, the boundedness of $\{y^{n_k}\}_{k=1}^\infty$ implies that there exists a subsequence $\{y^{n_{k_j}}\}_{j=1}^\infty$ of $\{y^{n_k}\}_{k=1}^\infty$ such that $y^{n_{k_j}}\to x^*\in\R^k.$ 
From this and the relation (\ref{eq5}), the demi-closedness principle of $T$ implies that $x^*\in\fix T.$ 
According to the fact that $\eta_n\to0$
and the boundedness of $\{f^\prime(x^{n_{k_j}})+\theta_{n_{k_j}}h^\prime(x^{n_{k_j}})\}_{j=1}^\infty$, it holds that 
 $\|x^{n_{k_j}}-y^{n_{k_j}}\|\to0$.
Since $y^{n_{k_j}}\to x^*$, we thus have $x^{n_{k_j}}\to x^*$.
Moreover, since the functions $f$ and $-g$ are both convex, the function $\frac{f}{g}$ is continuous. 
By these and the relation (\ref{eq6}), we obtain that
\begin{equation*}
\frac{f(x^*)}{g(x^*)}=\lim\limits_{j\rightarrow\infty} \frac{f(x^{n_{k_j}})}{g(x^{n_{k_j}})} \le \frac{f(z)}{g(z)},
\end{equation*}
for all $z\in \fix T$. That is, $x^*\in \mathcal{S}.$ Finally, we shall prove that $x^{n_k}\to x^*\in \mathcal{S}$. Since $\{x^{n_k}\}_{k=1}^\infty$ is bounded, it is enough to prove that there is no subsequence $\{x^{n_{k_l}}\}_{l=1}^\infty$ of  $\{x^{n_k}\}_{k=1}^\infty$  such that $\mathop {\lim }\limits_{l \to \infty }x^{n_{k_l}}=\bar{x}\in \mathcal{S}$ and $\bar{x}\ne x^*.$  Assuming that the aforementioned statement is not true, the existence of the limit of $\|x^n-z\|$ ($z\in\fix T$) and Opial's theorem  yield that
\begin{eqnarray*}
	\lim\limits_{n\to\infty}\|x^n-x^*\| =
 \lim\limits_{j\to\infty}\|x^{n_{k_j}}-x^*\|&<&\lim\limits_{j\to\infty}\|x^{n_{k_j}}-\bar{x}\|\\
 &=& \lim\limits_{k\to\infty}\|x^{n_{k}}-\bar{x}\|\\
 &=& \lim\limits_{l\to\infty}\|x^{n_{k_l}}-\bar{x}\|\\
 &<&\lim\limits_{l\to\infty}\|x^{n_{k_l}}-x^*\| = \lim\limits_{n\to\infty}\|x^n-x^*\|.
\end{eqnarray*}
This brings a contradiction. Therefore, there exists a subsequence of $\{x^n\}_{n=1}^\infty$ that converges to a point in the solution set $\mathcal{S}$.

		
			\textbf{Case 2.}
		Assume that there is a point $z\in \fix T$ and a subsequence $\{x^{n_k}\}_{k=1}^\infty$ of $\{x^n\}_{n=1}^\infty$ such that $\|x^{n_k}-z\|^2<\|x^{n_k+1}-z\|^2$ for all $k\in\N$.
    Let $\{\nu(n)\}_{n=1}^\infty$ be defined as in Lemma \ref{paul}. Then, for all $n\ge n_0$, it holds that
				\begin{equation}\label{eq7}
				\|x^{\nu(n)}-z\|^2\le \|x^{\nu(n)+1}-z\|^2,
				\end{equation}
				and
				\begin{equation}\label{eq8}
				\|x^{n}-z\|^2\le \|x^{\nu(n)+1}-z\|^2.
				\end{equation}
		By putting $B_2:=\sup\limits_{n \ge 1}\left|\frac{f(z)}{g(z)} - \theta_{\nu(n)}\right|$, we obtain from Lemma \ref{lm-1-bdd} and the inequality (\ref{eq7}) that
\begin{eqnarray*}
    \rho\|Ty^{\nu(n)}-y^{\nu(n)}\|^2
    &\le& \|x^{\nu(n)}-z\|^2 - \|x^{\nu(n)+1}-z\|^2 \\
    &&+ 2g(z)\eta_{\nu(n)}\left(\frac{f(z)}{g(z)}-\theta_{\nu(n)}\right) + B\eta_{\nu(n)}^2\\
    &\le&  2g(z)B_2\eta_{\nu(n)}+ B\eta_{\nu(n)}^2.
\end{eqnarray*}
This together with the assumption $\sum\limits_{n=1}^{\infty}\eta_{\nu(n)}^2 < \infty$
implies that 
\begin{equation}\label{eq9}
\lim\limits_{n\rightarrow\infty}\|Ty^{\nu(n)}-y^{\nu(n)}\|=0.
\end{equation}
Additionally, by utilizing Lemma \ref{lm-1-bdd} and the inequality (\ref{eq7}), we have 
\begin{equation*}  2g(z)\eta_{\nu(n)}\left(\theta_{\nu(n)}-\frac{f(z)}{g(z)}\right) \le \|x^{\nu(n)}-z\|^2 - \|x^{\nu(n)+1}-z\|^2 + B\eta_{\nu(n)}^2 \le B\eta_{\nu(n)}^2.
\end{equation*}
Since $\{\eta_{\nu(n)}\}_{n=1}^\infty\subset(0,\infty)$ and $\eta_{\nu(n)}\to0$, we deduce that
\begin{equation}\label{eq10}
\limsup\limits_{n\rightarrow\infty}\frac{f(x^{\nu(n)})}{g(x^{\nu(n)})}\le\frac{f(z)}{g(z)}.
\end{equation}
Choose a subsequence $\{x^{\nu(n_k)}\}_{k=1}^\infty$ of $\{x^{\nu(n)}\}_{n=1}^\infty$ 
arbitrarily. From the above, we have 
\begin{equation}\label{eq11}
\limsup\limits_{k\rightarrow\infty}\frac{f(x^{\nu(n_k)})}{g(x^{\nu(n_k)})}\le \limsup\limits_{n\rightarrow\infty}\frac{f(x^{\nu(n)})}{g(x^{\nu(n)})}\le\frac{f(z)}{g(z)}.
\end{equation}
Furthermore, since the sequence $\{y^{\nu(n_k)}\}_{k=1}^\infty$ is bounded, there exists a subsequence $\{y^{\nu(n_{k_j})}\}_{j=1}^\infty$ of $\{y^{\nu(n_k)}\}_{k=1}^\infty$ such that $y^{\nu(n_{k_j})}\to x^*\in\R^k.$ By invoking this and the relation (\ref{eq9}), the demi-closedness principle of $T$ yields that $x^*\in\fix T.$ 
Due to the fact that 
$\eta_{\nu(n)}\to0$ and the boundedness of $\{f^\prime(x^{n_{k_j}})+ \theta_{n_{k_j}}h^\prime(x^{n_{k_j}})\}_{j=1}^\infty$, we derive that 
$\|x^{\nu(n_{k_j})}-y^{\nu(n_{k_j})}\|\to 0$.
As $y^{\nu(n_{k_j})}\to x^*$, it follows that $x^{\nu(n_{k_j})}\to x^*$. This, together with the lower semicontinuity of $\frac{f}{g}$, the boundedness of $\{\theta_{\nu(n_{k_j})}\}$ and the inequality (\ref{eq11}), implies that
\begin{equation*}
\frac{f(x^*)}{g(x^*)}\le\liminf\limits_{j\rightarrow\infty} \frac{f(x^{\nu(n_{k_j})})}{g(x^{\nu(n_{k_j})})} \le \limsup\limits_{j\rightarrow\infty} \frac{f(x^{\nu(n_{k_j})})}{g(x^{\nu(n_{k_j)}})}\le \frac{f(z)}{g(z)},
\end{equation*}
for all $z\in \fix T$. Thus, we have $x^*\in \mathcal{S}.$ Now, in view of (\ref{eq8}), since $x^{\nu(n_{k_j})}\to x^*$, we have
\begin{equation*}
0\le\limsup\limits_{j\rightarrow\infty}\|x^{n_{k_j}}-x^*\| \le\limsup\limits_{j\rightarrow\infty}\|x^{\nu(n_{k_j})}-x^*\|=0.
\end{equation*}
Hence, $ x^{n_{k_j}}\to x^*\in \mathcal{S}$. Therefore, in both cases, we obtain that there exists a subsequence of $\{x^n\}_{n=1}^\infty$ that converges to a point in $\mathcal{S}$. The proof is  complete.   
\end{proof}

   \begin{remark}\label{remark-algor1} 
According to the assumptions that $T$ is strongly quasi-nonexpansive and the sequence $\{x^n\}_{n=1}^\infty$ is bounded, one can view $T$ as a composition of finitely many strongly quasi-nonexpansive operators \cite[Theorem 2.1.48]{C12}. For instance, $T:=P_{C_m}T_{m-1}\cdots T_1$, where each $T_i$ (for  $i=1,2,\ldots,m-1$) is strongly quasi-nonexpansive,  and $P_{C_m}$ is the metric projection onto a bounded, closed and convex subset $C_m$. This ensures the boundedness of the sequence $\{x^n\}_{n=1}^\infty$. Additionally, in practical situations, one can construct a closed ball with a sufficiently large radius to control the boundedness of the generated sequences. For further discussion, see references \cite{I11,I15,IP16,PN21}.
   \end{remark}
			

\subsection{Convergence of the Strongly Convex Case}\label{sect-3.2}

As we have presented in the previous section, the subsequential convergence of the FSSM is obtained by assuming the boundedness of the sequence $\{x^n\}_{n=1}^\infty$. In this section, we will show that if $f$ is assumed to be $\sigma$-strongly convex for some $\sigma>0$, then the assumption that  $\{x^n\}_{n=1}^\infty$ is bounded can be completely removed.  To proceed with this, we propose the modified version of the FSSM for solving Problem (\ref{main-FP}) in the following: 	
		\vskip2mm
		\begin{algorithm}[H]
			\SetAlgoLined
			\vskip2mm
			\textbf{Initialization}:Choose a stepsize sequence $\{\eta_n\}_{n=1}^\infty\subset(0,\infty)$ and take an initial point $x^1\in \R^k$ with $g(x_1)>0$. Set $\theta_1=\frac{f(x^1)}{g(x^1)}$. 

\textbf{Iterative step}: For a current iterate $x^n\in \im T$ ($n\in\N$), calculate as follows:
\vspace{0.3cm}

\textbf{Step 1.} Compute $\theta_n=\frac{f(x^n)}{g(x^n)}$. 

\vspace{0.3cm}

\textbf{Step 2.} Compute subgradients $f^\prime (x^n)\in\partial f(x^n)$ and $h^\prime (x^n)\in\partial (-g)(x^n)$ and calculate the next iterate $x^{n+1}\in \im T$as 
			\begin{equation*}
     x^{n+1}=T\left(x^n-\eta_n \left(\frac{f^\prime (x^n) + \theta_nh^\prime(x^n)}{\max\{1,\|f^\prime (x^n) + \theta_nh^\prime(x^n)\|\}}\right)\right).
\end{equation*}
Update $n:=n+1$ and go to \textbf{Step 1.}
			\caption{Adaptive Fixed-Point Subgradient Splitting Method (AFSSM)}
			\label{algor2}
			\vskip2mm
			
		\end{algorithm}
		
		\vskip2mm
		
Before conducting the convergence analysis of the AFSSM, we make the following remark.

  	\begin{remark}  
  	 Notice that  the AFSSM  is a modified version of the FSSM in the case where the term $f^\prime (x^n) + \eta_n\theta_n h^\prime (x^n)$ is restricted in a unit ball centered at the origin in the attempt to remove the bounded assumption of the sequence $\{x^n\}_{n=1}^\infty$. 		  	
     \end{remark}

Now, let $\{x^n\} _{n = 1}^\infty$ be a sequence generated by the AFSSM. For the sake of simplicity, we denote 
\begin{equation*}
   u^n:=x^n-\eta_n \left(\frac{f^\prime (x^n) + \theta_nh^\prime(x^n)}{\max\{1,\|f^\prime (x^n) + \theta_nh^\prime(x^n)\|\}}\right), \quad \text{for all $n\in\N$.}
\end{equation*}

The following lemma demonstrates significant properties that will be needed in proving convergence analysis of the AFSSM in the setting of Problem (\ref{main-FP}). 

 \begin{lemma}\label{2-lm-1} 
Let $\{x^n\} _{n = 1}^\infty$ be a sequence generated by the AFSSM. Suppose that $\sum\limits_{n=1}^{\infty}\eta_n=\infty$ and $\sum\limits_{n=1}^{\infty}\eta_n^2 < \infty$. If the function $f$ is $\sigma$-strongly convex, then for any $z\in \fix T$, the following  hold:
\begin{itemize}
    \item[(i)]
    The sequence $\{x^n\}_{n=1}^\infty$ is bounded;
    \item[(ii)] 
     $\|x^{n+1}-z\|^2 \le \|x^n-z\|^2  + \frac{2g(z)\eta_n}{\max\{1,\|f^\prime(x^n) + \theta_nh^\prime(x^n)\|\}}\left(\frac{f(z)}{g(z)}-\theta_n\right) + \eta_n^2  - \rho\|Tu^n -u^n\|^2$;
     
   \item[(iii)] 
$\liminf\limits_{n\rightarrow\infty}\theta_n \le \frac{f(z)}{g(z)}.$
\end{itemize}

    \end{lemma}

\begin{proof}
    Let $z\in\fix T$. Assume that $f$ is $\sigma$-strongly convex for some $\sigma>0$. We start with proving (i). By the $\rho$-strong quasi-nonexpansivity of $T$, we obtain
{\small \begin{eqnarray}\label{2-eq1}
    \|x^{n+1}-z\|^2 
    &\le& \left\|x^n-\eta_n \left(\frac{f^\prime (x^n) + \theta_nh^\prime(x^n)}{\max\{1,\|f^\prime (x^n) + \theta_nh^\prime(x^n)\|\}}\right)-z\right\|^2 - \rho\|Tu^n -u^n\|^2\nonumber\\
&=&\|x^n-z\|^2 + \eta_n^2 - \frac{2\eta_n}{\max\{1,\|f^\prime (x^n) + \theta_nh^\prime(x^n)\|\}}\langle x^n-z, f^\prime(x^n) + \theta_n h^\prime(x^n)\rangle \nonumber\\
&&- \rho\|Tu^n -u^n\|^2.
     \end{eqnarray} }
Since $\rho>0$, we thus have
{\small \begin{equation}\label{2-eq2}
    \|x^{n+1}-z\|^2 \le \|x^n-z\|^2 + \eta_n^2 - \frac{2\eta_n}{\max\{1,\|f^\prime (x^n) + \theta_nh^\prime(x^n)\|\}}\langle x^n-z, f^\prime(x^n) + \theta_n h^\prime(x^n)\rangle. 
\end{equation} }
Now, we set 
\begin{equation*}
    \Phi_n:=\|x^n-z\|^2 - \sum\limits_{k = 1}^{n-1}\eta_k^{2}, \quad \text{for all $n\in\N.$}
\end{equation*}
Then, we have 
\begin{equation}\label{2-eq3}
    \Phi_{n+1}-\Phi_{n}+\frac{2\eta_n}{\max\{1,\|f^\prime (x^n) + \theta_nh^\prime(x^n)\|\}}\langle x^n-z, f^\prime(x^n) + \theta_n h^\prime(x^n)\rangle\le0.
\end{equation}
To obtain the boundedness of the sequence $\{x^n\}_{n=1}^\infty$, we divide the proof into 2 cases:

\indent\textbf{Case 1.} Assume that there is an ${n_0}\in\mathbb{N}$ such that $\{\Phi_n\}_{n\geq n_0}$ is nonincreasing. Then, for all $n\ge n_0$, we have
$ \|x^n-z\|^2 -  \sum\limits_{k = 1}^{n-1}\eta_k^{2} \le \Phi_{n_0}.$ Due to the boundedness of the sequence $\{\eta_n\}_{n=1}^\infty$, we obtain that the sequence $\{\|x^n-z\|^2\}_{n=1}^\infty$ is bounded. This immediately implies that the sequence $\{x^n\}_{n=1}^\infty$ is bounded.

\textbf{Case 2.}
		Assume that there is a subsequence $\{\Phi_{n_k}\}_{k=1}^\infty$ of $\{\Phi_{n}\}_{n=1}^\infty$ such that $\Phi_{n_k}\le\Phi_{n_k+1}$ for all $k\in\N$.
    Let $\{\nu(n)\}_{n=1}^\infty$ be defined as in Lemma \ref{paul}. Then, for all $n\ge n_0$, it holds that
				\begin{equation}\label{2-eq4}
				\Phi_{\nu(n)}\le \Phi_{\nu(n)+1} ,
				\end{equation}
				and
			\begin{equation}\label{2-eq5}
				\Phi_{n} \le \Phi_{\nu(n)+1}.
				\end{equation}
In view of (\ref{2-eq3}), we obtain from (\ref{2-eq4}) that 
\begin{eqnarray*}
    &&\frac{2\eta_{\nu(n)}}{\max\{1,\|f^\prime (x^{\nu(n)}) + \theta_{\nu(n)}h^\prime(x^{\nu(n)})\|\}}\langle x^{\nu(n)}-z, f^\prime(x^{\nu(n)}) + \theta_{\nu(n)}h^\prime(x^{\nu(n)})\rangle \\
    &\le& \Phi_{\nu(n)} - \Phi_{\nu(n)+1} \le 0.
\end{eqnarray*}
As $\{\eta_{\nu(n)}\}_{n=1}^\infty\subset(0,\infty)$, we deduce that
\begin{equation}\label{2-eq6}
   \langle x^{\nu(n)}-z, f^\prime(x^{\nu(n)})+ \theta_{\nu(n)}h^\prime(x^{\nu(n)})\rangle \le 0.
\end{equation}
Note that, by combining the subgradient inequalities of $f$ and $-g$, we obtain
\begin{align}
    \langle f^\prime(x^{\nu(n)}) + \theta_{\nu(n)}h^\prime(x^{\nu(n)}), z-x^{\nu(n)}\rangle 
     &\le f(z) - \theta_{\nu(n)}g(z).\nonumber
     \end{align}
This and the inequality (\ref{2-eq6}) yield that 
\begin{equation*}
    \theta_{\nu(n)}g(z) - f(z) \le \langle x^{\nu(n)} -z, f^\prime(x^{\nu(n)}) + \theta_{\nu(n)}h^\prime(x^{\nu(n)})\rangle \le 0,
\end{equation*}
which immediately implies that 
\begin{equation}\label{2-eq7}
\theta_{\nu(n)}\le\frac{f(z)}{g(z)}.
\end{equation}
Furthermore, since $f$ is $\sigma$-strongly convex and $-g$ is convex, we obtain that $f - \theta_{\nu(n)}g$ is also $\sigma$-strongly convex \cite[Lemma 5.20]{B17}. By invoking Lemma \ref{lemma-strongconvex},
the inequality (\ref{2-eq6}), the triangle inequality and the relation (\ref{2-eq7}), we obtain
{\small \begin{align}
   \sigma\|x^{\nu(n)}-z\|^2 &\le \langle (f^\prime(x^{\nu(n)}) + \theta_{\nu(n)}h^\prime(x^{\nu(n)})) - (f^\prime(z) + \theta_{\nu(n)}h^\prime(z)),  x^{\nu(n)}-z\rangle\nonumber\\
&\le \|f^\prime(z) + \theta_{\nu(n)}h^\prime(z)\|\|x^{\nu(n)}-z\|\nonumber\\
&\le\left(\|f^\prime(z)\|+\frac{f(z)}{g(z)}\|h^\prime(z)\|\right)\|x^{\nu(n)}-z\|,\nonumber
     \end{align}  }
     which leads to
    \begin{equation*}
        \|x^{\nu(n)}-z\|\le\frac{1}{\sigma}\left(\|f^\prime(z)\|+\frac{f(z)}{g(z)}\|h^\prime(z)\|\right).
    \end{equation*}
Thus, the sequence $\{\|x^{\nu(n)}-z\|\}_{n=1}^\infty$ is bounded. In addition, we note that
\begin{equation*}
\Phi_{\nu(n)+1}:=\|x^{\nu(n)+1}-z\|^2 - \sum\limits_{k = 1}^{{\nu(n)}}\eta_k^{2}\le \|x^{\nu(n)+1}-z\|^2,
\end{equation*}
which implies that the sequence $\{\Phi_{\nu(n)+1}\}_{n=1}^\infty$ is bounded. Finally, according to the inequality (\ref{2-eq5}), we deduce that the sequence $\{\Phi_{n}\}_{n=1}^\infty$  is bounded and hence the sequence $\{x^n\}_{n=1}^\infty$ is also bounded.

Next, by invoking the subgradient inequalities of $f$ and $-g$,  the inequality (\ref{2-eq1}) becomes  
{\small \begin{eqnarray}
    \|x^{n+1}-z\|^2
    &\le& \|x^n-z\|^2 + \eta_n^2   -  \rho\|Tu^n -u^n\|^2 \nonumber\\
    &&+ \frac{2\eta_n}{\max\{1,\|f^\prime (x^n) + \theta_nh^\prime(x^n)\|\}}(f(z)-\theta_ng(z) + \theta_ng(x^n)-f(x^n)),\nonumber
\end{eqnarray} }
from which we deduces (ii).

 Finally, we shall prove (iii). 
Due to the boundedness of $\{\max\{1,\|f^\prime (x^n) + \theta_nh^\prime(x^n)\|\}\}_{n=1}^\infty$, there exists a real number $D\ge1$ such that $\max\{1,\|f^\prime (x^n) + \theta_nh^\prime(x^n)\|\}\le D$ for all $n\ge1$. Thus, (ii) becomes 
\begin{eqnarray*}
    2g(z)\eta_n\left(\theta_n - \frac{f(z)}{g(z)}\right)      &\le& D(\|x^n-z\|^2 - \|x^{n+1}-z\|^2 + \eta_n^2).
\end{eqnarray*}
   By summing the above inequality from $n=n_0$ to infinity, we obtain that
    \begin{eqnarray*}
  	\sum\limits_{n = n_0}^\infty\eta_n\left(\theta_n - \frac{f(z)}{g(z)}\right)
&\le& \frac{D}{2g(z)}\|x^{n_0}-z\|^2 +\frac{D}{2g(z)}\sum\limits_{n = n_0}^\infty \eta_n^2, 
   \end{eqnarray*}
   which implies that 
  	$\sum\limits_{n =1}^\infty\eta_n\left(\theta_n - \frac{f(z)}{g(z)}\right)< \infty.$
Then, by following the same arguments as in Theorem \ref{thm-1-obj}(ii), we obtain (iii) immediately. The proof is complete. 
\end{proof}


The following is the convergence theorem of the AFSSM in the setting of Problem (\ref{main-FP}).
		
	\begin{theorem}\label{2-main-thm}
			Let $\{x^n\} _{n = 1}^\infty$ be a sequence generated by the AFSSM. Suppose that $\sum\limits_{n=1}^{\infty}\eta_n=\infty$ and $\sum\limits_{n=1}^{\infty}\eta_n^2 < \infty$.
 			 If the operator $T$ satisfies the demi-closedness principle and the function $f$ is $\sigma$-strongly convex, then the sequence $\{x^n\}_{n=1}^\infty$ converges to the unique solution to Problem (\ref{main-FP}).
		\end{theorem}

\begin{proof}
Suppose that $T$ satisfies the demi-closedness principle and $f$ is $\sigma$-strongly convex for some $\sigma>0$. 
Since $f$ is $\sigma$-strongly convex and nonnegative, and $g$ is concave, positive and bounded above by $M$, it follows from \cite[Proposition 4.1]{ILMY24} that the function $\frac{f}{g}$ is $\frac{\sigma}{M}$-strongly quasi-convex\footnote{A function $f:\R^k\to\R$ is said to be {\it $\rho$-strongly quasi-convex} \cite{L22}, where $\rho>0$, if $f(\alpha x + (1-\alpha)y)\le\max\{f(x),f(y)\} - \alpha(1-\alpha)\frac{\rho}{2}\|x-y\|^2$ for all $x,y\in\R^k$ and $\alpha\in[0,1]$.}. Furthermore, since $\fix T$ is closed and convex \cite[Proposition 2.6(ii)]{BC01-P}, it follows from \cite[Corollary 3]{L22} that Problem (\ref{main-FP}) has a unique minimizer. Therefore, the solution set $\mathcal{S}$ contains exactly one point, which we denote by $x^*$. 
The proof follows the same structure as in Theorem \ref{main-thm-bdd}, except that we use Lemma \ref{2-lm-1} instead of Lemma \ref{lm-1-bdd} to proceed. To be precise, we divide the proof into two cases, as in Theorem \ref{main-thm-bdd}, with some additional details as follows:

\indent\textbf{Case 1.} Suppose that there is an ${n_0}\in\mathbb{N}$ such that $\{\|x^n-z\|^2\}_{n\geq n_0}$ is nonincreasing for all $z\in \fix T$. Then, $\mathop {\lim }\limits_{n \to \infty }{\|x^n-z\|^2}$ exists. 
Set $P_1:=\sup\limits_{n\ge 1}\left|\frac{f(z)}{g(z)} - \theta_n\right|$. Lemma \ref{2-lm-1}(ii) implies that 
\begin{equation*}
\rho\limsup\limits_{n\rightarrow\infty}\|Tu^n - u^n\|
  \le\lim\limits_{n\rightarrow\infty}\|x^{n}-z\|^2 - \lim\limits_{n\rightarrow\infty}\|x^{n+1}-z\|^2 + 2g(z)P_1\lim\limits_{n\rightarrow\infty}\eta_n + \lim\limits_{n\rightarrow\infty}\eta_n^2.
\end{equation*}
The assumption $\sum\limits_{n=1}^{\infty}\eta_n^2 < \infty$ and the fact that $\rho>0$ lead to 
$\|Tu^n-u^n\|\to0.$
Then, by following the same arguments as in {\bf Case 1} of Theorem \ref{main-thm-bdd}, we obtain that there exists a subsequence $\{x^{n_k}\}_{k=1}^\infty$ of $\{x^{n}\}_{n=1}^\infty$ such that $x^{n_k}\to x^*\in \mathcal{S}$.
    Furthermore, since the solution set $\mathcal{S}$ contains exactly one point, we conclude that the sequence $\{x^n\}_{n=1}^\infty$ converges to the unique solution $x^*$.

 \textbf{Case 2.}
		Suppose that there is a point $z\in \fix T$ and a subsequence $\{x^{n_k}\}_{k=1}^\infty$ of $\{x^n\}_{n=1}^\infty$ such that $\|x^{n_k}-z\|^2<\|x^{n_k+1}-z\|^2$ for all $k\in\N$.
    Let $\{\nu(n)\}_{n=1}^\infty$ be defined as in Lemma \ref{paul}. Then, for all $n\ge n_0$, it holds that
				\begin{equation}\label{eq7-sc}
				\|x^{\nu(n)}-z\|^2\le \|x^{\nu(n)+1}-z\|^2,
				\end{equation}
				and
				\begin{equation}\label{eq8-sc}
				\|x^{n}-z\|^2\le \|x^{\nu(n)+1}-z\|^2.
				\end{equation}
		By setting $P_2:=\sup\limits_{n \ge 1}\left|\frac{f(z)}{g(z)} - \theta_{\nu(n)}\right|$, we obtain from Lemma \ref{2-lm-1}(ii) and the inequality (\ref{eq7-sc}) that
\begin{eqnarray*}
    \rho\|Tu^{\nu(n)}-u^{\nu(n)}\|^2
    &\le& \|x^{\nu(n)}-z\|^2 - \|x^{\nu(n)+1}-z\|^2 \\
    &&
    + \frac{2g(z)P_2\eta_{\nu(n)}}{\max\{1,\|f^\prime (x^{\nu(n)}) + \theta_{\nu(n)}h^\prime(x^{\nu(n)})\|\}} + \eta_{\nu(n)}^2\\
    &\le&  2g(z)P_2\eta_{\nu(n)}+ \eta_{\nu(n)}^2.
\end{eqnarray*}
This together with $\sum\limits_{n=1}^{\infty}\eta_{\nu(n)}^2 < \infty$
implies that 
$\|Tu^{\nu(n)}-u^{\nu(n)}\|\to0.$
Since $\{\max\{1,\|f^\prime (x^{\nu(n)}) + \theta_{\nu(n)}h^\prime(x^{\nu(n)})\|\}\}_{n=1}^\infty$ is bounded, there exists a real number $L\ge1$ such that $\max\{1,\|f^\prime (x^{\nu(n)}) + \theta_{\nu(n)}h^\prime(x^{\nu(n)})\|\}\le L$ for all $n\ge1$. From this and the inequality (\ref{eq7-sc}), we deduce from Lemma \ref{2-lm-1}(ii) that 
\begin{eqnarray*}
    \eta_{\nu(n)}\left(\theta_{\nu(n)} - \frac{f(z)}{g(z)}\right) 
     \le \frac{L}{2g(z)}(\|x^{\nu(n)}-z\|^2 - \|x^{\nu(n)+1}-z\|^2 + \eta_{\nu(n)}^2) \le \frac{L}{2g(z)}\eta_{\nu(n)}^2.
\end{eqnarray*}
The facts that $\{\eta_{\nu(n)}\}_{n=1}^\infty\subset(0,\infty)$, $g(z)>0$, and 
$\eta_{\nu(n)}\to0$ imply that 
\begin{equation*}
\limsup\limits_{n\rightarrow\infty}\frac{f(x^{\nu(n)})}{g(x^{\nu(n)})}\le\frac{f(z)}{g(z)}.
\end{equation*}
Then, by following the same arguments as in {\bf Case 2} of Theorem \ref{main-thm-bdd}, we obtain that there exists a subsequence $\{x^{\nu(n_{k})}\}_{n=1}^\infty$ of $\{x^{\nu(n)}\}_{n=1}^\infty$ such that $x^{\nu(n_k)}\to x^*$.
    Since the solution set $\mathcal{S}$ consists of one point, we have that $x^{\nu(n)}\to x^*$. In view of the relation (\ref{eq8-sc}), we get
    \begin{equation*} 0\le\limsup\limits_{n\rightarrow\infty}\|x^n - x^*\|\le \limsup\limits_{n\rightarrow\infty}\|x^{\nu(n)} - x^*\|=0, \end{equation*}
    which implies that the sequence $\{x^n\}_{n=1}^\infty$ converges to the unique solution $x^*$. The proof is complete.  
\end{proof}

\begin{remark} Some essential remarks are as follows:
\begin{itemize} 

\item[(i)] Since the AFSSM does not require the boundedness of the sequence $\{x^n\}_{n=1}^\infty$, as in the FSSM, one can view the operator $T$ as a composition or simultaneous application of strongly quasi-nonexpansive operators. That is, $T$ can either be the sequential composition  $T:=T_{m}T_{m-1}\cdots T_1$, or the weighted sum  $T:=\sum_{i=1}^m\omega_i T_i$, where $T_i$ is a strongly quasi-nonexpansive operator for all $i=1,2,\ldots,m$.

 \item[(ii)] It is important to note that the FSSM and the AFSSM are based on different assumptions regarding the boundedness of the sequence \( \{x^n\}_{n=1}^\infty \) and the strong convexity of the function \( f \), which have significant implications for their convergence behavior. The FSSM assumes that the sequence \( \{x^n\}_{n=1}^\infty \) is bounded, a condition that is crucial for ensuring subsequential convergence. In contrast, the AFSSM removes the boundedness assumption on \( \{x^n\}_{n=1}^\infty \), but requires \( f \) to be strongly convex, which guarantees a unique solution and faster convergence. However, this assumption limits the algorithm's applicability to problems where strong convexity holds. Additionally, the step size in the AFSSM must be restricted to \( \frac{\eta_n}{\max\{1,\|f^\prime (x^n) + \theta_n h^\prime(x^n)\|\}} \), in contrast to the simpler step size \( \eta_n \) in the FSSM. This restriction on the step size may affect the speed of convergence of the AFSSM. These differences highlight a trade-off between the two algorithms: the FSSM is more broadly applicable due to its convexity assumption, but the boundedness requirement can sometimes be restrictive. On the other hand, the AFSSM does not require the boundedness assumption but instead relies on strong convexity and a restriction on the step size \( \eta_n \), which can impact its convergence rate.

\end{itemize}
   \end{remark}



\section{Incremental Fixed-Point Subgradient Splitting Method}\label{sect-incremental}

Within the context of fractional programming, there is a natural connection to the problem that considers the minimization of a sum of finitely many ratios. In this section, we consider the following problem:

 \begin{equation} \label{P-Incremental}\displaystyle
 \tag{SP}
	\begin{array}{ll}
\mathrm{minimize}\hskip0.2cm\indent  \displaystyle \sum\limits_{i=1}^{m}\frac{f_i(x)}{g_i(x)}\\
	\mathrm{subject \hskip0.2cm to}\indent \displaystyle  x\in\bigcap\limits_{i = 1}^m \fix{{T_i}},
	\end{array}  
 \end{equation}
where the following assumptions hold:
\begin{itemize}
\item[(B1)] $T_i:\R^k\to\R^k$, $i=1,2,\ldots,m$, are firmly nonexpansive operators.
\item [(B2)] $f_i:\R^k\to\R$, $i=1,2,\ldots,m$,  are convex, and each $f_i(x)\geq0$ for all   $x\in\bigcap_{i=1}^m \im T_i$;
\item[(B3)] $g_i:\R^k\to\R$, $i=1,2,\ldots,m$, are concave, and there exist constants $M,N>0$ such that $m+N\le M$ and $0<N\le g_i(x)\leq M$ for all $i=1,2,\ldots,m$ and for all $x\in\bigcap_{i=1}^m \im T_i$.
\end{itemize}
We define $\theta^*:=\min\limits_{x\in \bigcap_{i=1}^m\fix T_i}\sum\limits_{i=1}^{m}\frac{f_i(x)}{g_i(x)}$,
and denote the solution set of Problem (\ref{P-Incremental}) as 
$\mathcal{Z}:=\left\{x^*\in\bigcap_{i=1}^m\fix T_i: \sum\limits_{i=1}^{m}\frac{f_i(x^*)}{g_i(x^*)} = \theta^* \right\},$
providing that the solution set $\mathcal{Z}$ is nonempty.

Motivated by the conventional ISM and Iiduka's I-ISM \cite{I16}, we present an iterative algorithm for solving the Problem (\ref{P-Incremental}). As in the previous section, we first introduce the necessary notations. Let $f_i^\prime(x)$ and $g_i^\prime(x)$, for all $i=1,2,\ldots,m$, denote the subgradients of $f_i$ and $g_i$ at $x\in\R^k$, respectively. For convenience, we define $h_i^\prime(x):= (-g_i)^\prime(x)$ for all $i=1,2,\ldots,m$, which represent the subgradient of $-g_i$ at $x$.
We now present the following method to solve Problem  (\ref{P-Incremental}):	
		
		\vskip2mm
        \begin{algorithm}[H]
			\SetAlgoLined
			\vskip2mm
			\textbf{Initialization}: Choose a stepsize sequence $\{\eta_n\}_{n=1}^\infty\subset(0,\infty)$ and take an initial point $x^1\in\R^k$ with $g_i(x_1)>0$ $i=1,2,\ldots,m$. 
			
			\textbf{Iterative step}: For a current iterate $x^n \in\R^k$ ($n\in\N$), calculate as follows:
			
    \vspace{0.3cm}
			\textbf{Step 1}. Set $x^{0,n}=x^n$.
            
             \vspace{0.3cm}
			\textbf{Step 2}. For any $i=1,2,\ldots,m$, compute 
 			\begin{equation*}
x^{i,n}=T_i\left(x^{i-1,n}-\eta_n f^\prime_i (x^{i-1,n}) - \eta_n\theta_{i,n} h^\prime_i (x^{i-1,n})\right),
\end{equation*}
where $\theta_{i,n}=\frac{f_i(x^n)}{g_i(x^n)}$, $f^\prime_i (x^{i-1,n})\in\partial f_i(x^{i-1,n})$ and $h_i^\prime (x^{i-1,n})\in\partial (-g_i)(x^{i-1,n})$.
\vspace{0.3cm}

			\textbf{Step 3}. Set $x^{n+1}=x^{m,n}$.
\vspace{0.3cm}
   
			Update $n:=n+1$ and go to \textbf{Step 1}.
			\caption{ Incremental Fixed-Point Subgradient Splitting Method (IFSSM) }
			\label{incremental-algor}
			\vskip2mm
			
		\end{algorithm}
		
		\vskip2mm

  Now, let $\{x^{i,n}\} _{n = 1}^\infty$ be a sequence generated by the IFSSM. For the sake of simplicity, we denote 
\begin{equation*}
     v^{i-1,n}:=x^{i-1,n}-\eta_n f^\prime_i (x^{i-1,n}) - \eta_n\theta_{i,n} h^\prime_i (x^{i-1,n}), \quad \text{for all $n\in\N$.}
\end{equation*}
Moreover, we define \( F: \mathbb{R}^k \to \mathbb{R} \) and \( g: \mathbb{R}^k \to \mathbb{R} \) at \( x \in \mathbb{R}^k \) as
\[
F(x) := \sum_{i=1}^{m} \frac{f_i(x)}{g_i(x)} \quad\text{and}\quad
g(x) := \sum_{i=1}^{m} g_i(x).
\]

\begin{assumption}\label{inc-assumption}
We assume the following conditions:
\begin{itemize}
    \item[(i)] The sequence $\{\eta_n\}_{n=1}^\infty\subset(0,\infty)$  satisfies $\sum\limits_{n=1}^{\infty}\eta_n=\infty$ and $\sum\limits_{n=1}^{\infty}\eta_n^2 < \infty$;
    \item[(ii)] The sequence $\{x^{n}\}_{n=1}^\infty$ is bounded.
\end{itemize}
    
\end{assumption}

Due to the boundedness of the sequence  $\{x^{n}\} _{n = 1}^\infty$, Proposition \ref{fact-BC17} implies the boundedness of the subgradients of $f_i$ and $-g_i$. Specifically, for each $i=1,2,\ldots,m$ and for all $n\in\N$, there exist scalars $L_i>0$ such that 
$$ \|u\|\le L_i \quad \text{for all} \quad u\in\partial f_i(x^n)\cup\partial f_i(x^{i-1,n}),$$ 
and there exist scalars $E_i>0$ such that $$\|w\|\le E_i \quad \text{for all} \quad w\in\partial (-g_i)(x^n)\cup\partial (-g_i)(x^{i-1,n}).$$


\begin{remark}
    According to the assumptions that each $T_i$ (for $i=1,2,\ldots,m$) is firmly nonexpansive and that the sequence $\{x^n\}_{n=1}^\infty$ is bounded, one can view each $T_i$ as the metric projection onto a bounded, closed and convex subset $P_{C_i}$. This ensures the boundedness of the sequence $\{x^n\}_{n=1}^\infty$.
    Furthermore, as discussed in Remark \ref{remark-algor1}(ii), in practical situations, one can also construct a closed ball with a sufficiently large radius to control the boundedness of the generated sequences. For further details, see references \cite{I11,I15,IP16,PN21}, particularly \cite{IP16}, where the author considers the incremental scheme.
\end{remark}


The following lemma indicates some essential properties in establishing the convergence analysis of the IFSSM in the setting of Problem (\ref{P-Incremental}).

   \begin{lemma}\label{incre-lm-1} 
    Let $\{x^n\} _{n = 1}^\infty$ be a sequence generated by the IFSSM. Suppose that Assumption \ref{inc-assumption} holds. Then, for any $z\in\bigcap_{i=1}^m\fix T_i$, 
    the following statements hold:
 \begin{itemize}
   \item[(i)] 
 $\|x^{n+1}-z\|^2 \le \|x^n-z\|^2 + 2g(z)\eta_n\left(F(z)-F(x^n)\right) +  4M\left(N+\frac{1}{N}\right)(L^2+E^2H^2)\eta_n^2 \\
       - \frac{1}{2}\sum\limits_{i=1}^{m}\|x^{i,n}-x^{i-1,n}\|^2,$
 where $H= \sup\limits_{n\ge 1}|\theta_{i,n}|$, $E=\sum\limits_{i=1}^{m}E_i$ and $L=\sum\limits_{i=1}^{m}L_i$;

   \item[(ii)] 
   $\liminf\limits_{n\rightarrow\infty}F(x^n) \le F(z).$
\end{itemize}

    \end{lemma}

   \begin{proof} 
Let $z\in\bigcap_{i=1}^m\fix T_i$ and suppose that $\{x^{n}\} _{n = 1}^\infty$ is bounded. 
First, we shall prove (i).
By the firmly nonexpansivity of each $T_i$, we have
\begin{eqnarray}
    \|x^{i,n}-z\|^2
     &\le&\langle T_i(v^{i-1,n}) - z, v^{i-1,n} -z\rangle\nonumber\\
    &=&\frac{1}{2}\left(\|x^{i,n}-z\|^2 + \|x^{i-1,n}-z- \eta_n(f^\prime_i (x^{i-1,n}) + \theta_{i,n} h^\prime_i (x^{i-1,n}))\|^2  \right)\nonumber\\
    &&- \frac{1}{2}(\|x^{i,n}-x^{i-1,n}+ \eta_n(f^\prime_i (x^{i-1,n}) + \theta_{i,n} h^\prime_i (x^{i-1,n}))\|^2).\nonumber
     \end{eqnarray}
By rearranging the above, we obtain
\begin{eqnarray}\label{incre-eq1}
    \|x^{i,n}-z\|^2 
 &\le& \|x^{i-1,n}-z\|^2 - \|x^{i,n}-x^{i-1,n}\|^2 - 2\eta_n\langle x^{i,n} -z, f_i^\prime(x^{i-1,n})\rangle \nonumber\\
&& - 2\eta_n\theta_{i,n}\langle x^{i,n} -z, h_i^\prime(x^{i-1,n})\rangle.
\end{eqnarray}
Note that 
{ \begin{eqnarray}\label{incre-eq2}
    - 2\eta_n\langle x^{i,n} -z, f_i^\prime(x^{i-1,n})\rangle 
    &=& 2\left\langle \sqrt{\frac{4M}{N}}(\eta_nf_i^\prime(x^{i-1,n})), \sqrt{\frac{N}{4M}}(x^{i-1,n}-x^{i,n})\right\rangle \nonumber \\
     &&+ 2\eta_n\langle f_i^\prime(x^{i-1,n}), z - x^{i-1,n}\rangle\nonumber \\
     &\le& 2\eta_n\langle f_i^\prime(x^{i-1,n}), z - x^{i-1,n}\rangle + \frac{4ML_i^2}{N}\eta_n^2 
+ \frac{N}{4M}\|x^{i,n}-x^{i-1,n}\|^2.\nonumber\\
\end{eqnarray}  }
 The subgradient inequality of each $f_i$ implies that
{ \begin{eqnarray}\label{incre-eq3}
    2\eta_n\langle f_i^\prime(x^{i-1,n}), z - x^{i-1,n}\rangle 
    &\le& 2\eta_n(f_i(z)-f_i(x^n)) + 2\langle \eta_nf_i^\prime(x^n), x^n-x^{i-1,n}\rangle\nonumber\\
    &=& 2\eta_n(f_i(z)-f_i(x^n)) + 2\left\langle\sqrt{4MN}\eta_nf_i^\prime(x^n), \frac{1}{\sqrt{4MN}}(x^n-x^{i-1,n})\right\rangle\nonumber\\
    &\le& 2\eta_n(f_i(z)-f_i(x^n)) + 4MNL_i^2\eta_n^2 
    + \frac{1}{4MN}\|x^n-x^{i-1,n}\|^2.
\end{eqnarray}  }
By the triangle inequality, we have 
\begin{eqnarray*}
    \|x^n-x^{i-1,n}\|&\le& \|x^{0,n}-x^{1,n}\| + \|x^{1,n}-x^{2,n}\|+\cdots + \|x^{i-2,n}-x^{i-1,n}\|\\
    &=& \sum\limits_{j=1}^{i-1}\|x^{j,n}-x^{j-1,n}\|
    \le \sum\limits_{i=1}^{m}\|x^{i,n}-x^{i-1,n}\|.
\end{eqnarray*}
This relation and Jensen's inequality deduce that 
\begin{equation}\label{incre-eq4}
    \|x^n-x^{i-1,n}\|^2 \le N^2\left(\frac{\sum\limits_{i=1}^{m}\|x^{i,n}-x^{i-1,n}\|}{N}\right)^2\le N\sum\limits_{i=1}^{m}\|x^{i,n}-x^{i-1,n}\|^2.
\end{equation}
By substituting this into the inequality (\ref{incre-eq3}), the inequality (\ref{incre-eq2}) turns into
{ \begin{eqnarray}\label{incre-eq5}
   - 2\eta_n\langle x^{i,n} -z, f_i^\prime(x^{i-1,n})\rangle 
   &\le& 2\eta_n(f_i(z)-f_i(x^n)) + 4M\left(N+\frac{1}{N}\right)L_i^2\eta_n^2 \nonumber\\
   &&+ \frac{1}{4M}\sum\limits_{i=1}^{m}\|x^{i,n}-x^{i-1,n}\|^2  + \frac{N}{4M}\|x^{i,n}-x^{i-1,n}\|^2.
\end{eqnarray}  }
Now, following a similar argument as in the derivation of (\ref{incre-eq2}) and (\ref{incre-eq3}), we obtain
  {   \begin{eqnarray}\label{incre-eq6}
    - 2\eta_n\theta_{i,n}\langle x^{i,n} -z, h_i^\prime(x^{i-1,n})\rangle 
      &\le& 2\eta_n\theta_{i,n}\langle h_i^\prime(x^{i-1,n}), z - x^{i-1,n}\rangle + \frac{4ME_i^2\theta_{i,n}^2}{N}\eta_n^2\nonumber\\
    &&+ \frac{N}{4M}\|x^{i,n}-x^{i-1,n}\|^2.
\end{eqnarray}}
Moreover, since each $g_i$ is concave, we have that $-g_i$ is convex for all $i=1,2,\ldots,m$. Then, the subgradient inequality of $-g_i$ and the relation (\ref{incre-eq4}) imply that
{ \begin{eqnarray}
    2\eta_n\theta_{i,n}\langle h_i^\prime(x^{i-1,n}), z - x^{i-1,n}\rangle 
     &\le& 2\eta_n\theta_{i,n}(-g_i(z)+g_i(x^n)) + 4MNE_i^2\theta_{i,n}^2\eta_n^2 \nonumber\\
    &&+\frac{1}{4M}\sum\limits_{i=1}^{m}\|x^{i,n}-x^{i-1,n}\|^2.\nonumber
\end{eqnarray}  }
Substituting the above inequality into the relation (\ref{incre-eq6}), we have 
{ \begin{eqnarray}\label{incre-eq7}
    -2\eta_n\theta_{i,n}\langle x^{i,n} -z, h_i^\prime(x^{i-1,n})\rangle 
 &\le& 2\eta_n\theta_{i,n}(-g_i(z)+g_i(x^n)) + 4M\left(N+\frac{1}{N}\right)E_i^2\theta_{i,n}^2\eta_n^2 \nonumber\\
   && +\frac{1}{4M}\sum\limits_{i=1}^{m}\|x^{i,n}-x^{i-1,n}\|^2 + \frac{N}{4M}\|x^{i,n}-x^{i-1,n}\|^2.
 \end{eqnarray}  }
By combining the inequalities (\ref{incre-eq5}) and (\ref{incre-eq7}) with the relation (\ref{incre-eq1}), we obtain 
{ \begin{eqnarray}\label{incre-eq8}
     \|x^{i,n}-z\|^2 
     &\le& \|x^{i-1,n}-z\|^2 + 2g_i(z)\eta_n\left(\frac{f_i(z)}{g_i(z)}-\theta_{i,n}\right) +  4M\left(N+\frac{1}{N}\right)(L_i^2+E_i^2H^2)\eta_n^2  \nonumber\\
    && 
    + \frac{1}{2M}\sum\limits_{i=1}^{m}\|x^{i,n}-x^{i-1,n}\|^2 -\left( 1-\frac{N}{2M}\right)\|x^{i,n}-x^{i-1,n}\|^2,\nonumber
\end{eqnarray} }
where  $H= \sup\limits_{n\ge 1}|\theta_{i,n}|$.
Then, summing the above inequality over $i=1,2,\ldots,m$ leads to
{ \begin{eqnarray}
    \|x^{n+1}-z\|^2 
      &\le& \|x^n-z\|^2 + 2g(z)\eta_n\left(F(z)-F(x^n)\right) +  4M\left(N+\frac{1}{N}\right)(L^2+E^2H^2)\eta_n^2 \nonumber\\
    &&     - \frac{1}{2}\sum\limits_{i=1}^{m}\|x^{i,n}-x^{i-1,n}\|^2,\nonumber
    \end{eqnarray}  }
which is of the form of (i) with $E=\sum\limits_{i=1}^{m}E_i$ and $L=\sum\limits_{i=1}^{m}L_i$.

Next, we will prove (ii). The proof follows similar arguments as Theorem \ref{thm-1-obj}(ii), only we use Lemma \ref{incre-lm-1}(i) instead of using Lemma \ref{lm-1-bdd}, to continue with the proof.   
\end{proof}

The following is the convergence theorem of the IFSSM in the setting of Problem (\ref{P-Incremental}).

\begin{theorem}\label{incre-main-thm}
    Let $\{x^n\} _{n = 1}^\infty$ be a sequence generated by the IFSSM. 
Suppose that Assumption \ref{inc-assumption} holds. Then, there exists a subsequence of the sequence $\{x^n\}_{n=1}^\infty$ that converges to a point $x^*$ in $\mathcal{Z}$.
\end{theorem}

\begin{proof}
    Let $z\in\bigcap_{i=1}^m\fix T_i$.
				Suppose that the sequence $\{x^{n}\} _{n = 1}^\infty$ is bounded.
			To prove the convergence of the theorem, we consider the following two cases
   of the sequence $\{\|x^n-z\|^2\}_{n=1}^\infty$ due to its behavior:
			
	\indent\textbf{Case 1.} Assume that there is an ${n_0}\in\mathbb{N}$ such that $\{\|x^n-z\|^2\}_{n\geq n_0}$ is nonincreasing for all $z\in\bigcap_{i=1}^m\fix T_i$. Then,  we have that $\mathop {\lim }\limits_{n \to \infty }{\|x^n-z\|^2}$ exists. 
    Put $F_1:=\sup\limits_{n\ge 1}\left|F(z) - F(x^n)\right|$. By Lemma \ref{incre-lm-1}(i), we obtain that
{ \begin{eqnarray}
\label{m-ic-1}
  \limsup\limits_{n\rightarrow\infty} \sum\limits_{i=1}^{m}\|x^{i,n}-x^{i-1,n}\|^2 &\le& 2\lim\limits_{n\rightarrow\infty}\|x^n-z\|^2 - 2\lim\limits_{n\rightarrow\infty}\|x^{n+1}-z\|^2+ 4g(z)F_1\lim\limits_{n\rightarrow\infty}\eta_n\nonumber\\
  &&+  8M\left(N+\frac{1}{N}\right)(L^2+E^2H^2)\lim\limits_{n\rightarrow\infty}\eta_n^2. 
\end{eqnarray} }
The assumption that $\sum\limits_{n=1}^{\infty}\eta_n^2 < \infty$ leads to
\begin{equation}\label{m-ic-5}
\lim\limits_{n\rightarrow\infty}\|x^{i,n}-x^{i-1,n}\|=0,
\end{equation}
for all $i=1,2,\ldots,m$. This, along with the nonexpansivity of each $T_i$ and the fact that $\eta_n\to0$, leads to
\begin{equation}\label{m-ic-52}
\lim\limits_{n\rightarrow\infty}\|T_ix^{i-1,n}-x^{i-1,n}\|=0,
\end{equation}
for all $i=1,2,\ldots,m$. 
Due to the boundedness of $\{x^{n}\}_{n=1}^\infty$ and Lemma \ref{incre-lm-1}(ii), there exists a subsequence $\{x^{n_k}\}_{k=1}^\infty$ of $\{x^{n}\}_{n=1}^\infty$ such that 
   \begin{equation}\label{m-ic-eq6}
\lim\limits_{k\rightarrow\infty} F(x^{n_k}) = \liminf\limits_{n\rightarrow\infty}F(x^{n}) \le F(z),
\end{equation}
for all  $z\in\bigcap_{i=1}^m\fix T_i$. 
Since $\{x^{n_k}\}_{k=1}^\infty$ is bounded, there exists a subsequence $\{x^{n_{k_j}}\}_{j=1}^\infty$ of $\{x^{n_k}\}_{k=1}^\infty$ such that $x^{n_{k_j}}\to x^*\in\R^k.$ We will now show that $x^*\in\bigcap_{i=1}^m\fix T_i$. 
Since $x^{n_{k_j}}\to x^*$, it follows that $x^{0,n_{k_j}}\to x^*$. By using this and the relation (\ref{m-ic-52}), the demi-closedness principle of $T_1$ implies that 
$x^*\in\fix T_1$. Furthermore, since $x^{0,n_{k_j}}\to x^*$ and by the relation (\ref{m-ic-5}), we obtain that $x^{1,n_{k_j}}\to x^*$. By applying this and the relation (\ref{m-ic-52}), the demi-closedness principle of $T_2$ yields that 
 $x^*\in\fix T_2$. Additionally, from $x^{1,n_{k_j}}\to x^*$  and the relation (\ref{m-ic-5}), we have that
$x^{2,n_{k_j}}\to x^*$. By using these same arguments for all $i$, we conclude that $x^*\in\fix T_i$ for all $i=1,2,\ldots,m$. Thus,  $x^*\in\bigcap_{i=1}^m\fix T_i$.
Furthermore, the continuity of $\frac{f}{g}$ and the relation (\ref{m-ic-eq6}), we obtain that
\begin{equation*}
F(x^*)=\lim\limits_{j\rightarrow\infty} F(x^{n_{k_j}}) \le F(z),
\end{equation*}
for all $z\in\bigcap_{i=1}^m\fix T_i$. That is, $x^*\in\mathcal{Z}$. Finally, we shall prove that $x^{n_k}\to x^*\in \mathcal{Z}$. Since $\{x^{n_k}\}_{k=1}^\infty$ is bounded, it is enough to prove that there is no subsequence $\{x^{n_{k_l}}\}_{l=1}^\infty$ of  $\{x^{n_k}\}_{k=1}^\infty$  such that $\mathop {\lim }\limits_{l \to \infty }x^{n_{k_l}}=\bar{x}\in \mathcal{Z}$ and $\bar{x}\ne x^*.$ To show this, the proof follows similar arguments to those used in {\textbf{Case 1}} of Theorem \ref{main-thm-bdd}.
Therefore, there exists a subsequence of $\{x^n\}_{n=1}^\infty$ that converges to a point in $\mathcal{Z}$.
   
		
			\textbf{Case 2.}
		Assume that there is a point $z\in \bigcap_{i=1}^m\fix T_i$ and a subsequence $\{x^{n_k}\}_{k=1}^\infty$ of $\{x^n\}_{n=1}^\infty$ such that $\|x^{n_k}-z\|^2<\|x^{n_k+1}-z\|^2$ for all $k\in\N$.
    Let $\{\nu(n)\}_{n=1}^\infty$ be defined as in Lemma \ref{paul}. Then, for all $n\ge n_0$, it holds that
				\begin{equation}\label{2-m-ic-eq7}
				\|x^{\nu(n)}-z\|^2\le \|x^{\nu(n)+1}-z\|^2,
				\end{equation}
				and
				\begin{equation}\label{2-m-ic-eq8}
				\|x^{n}-z\|^2\le \|x^{\nu(n)+1}-z\|^2.
				\end{equation}
		By putting $F_2:=\sup\limits_{n \ge 1}\left|F(z) - F(x^{\nu(n)})\right|$, Lemma \ref{incre-lm-1}(i) and the inequality (\ref{2-m-ic-eq7}) imply that
{\begin{eqnarray*}
\frac{1}{2}\sum\limits_{i=1}^{m}\|x^{i,\nu(n)}-x^{i-1,\nu(n)}\|^2 &\le& \|x^{\nu(n)}-z\|^2 - \|x^{{\nu(n)}+1}-z\|^2+ 2g(z)\eta_{\nu(n)}(F(z)-F(x^{\nu(n)}))\nonumber\\
  &&+  4M\left(N+\frac{1}{N}\right)(L^2+E^2H^2)\eta_{\nu(n)}^2.\\
    &\le& 2g(z)F_2\eta_{\nu(n)} +  4M\left(N+\frac{1}{N}\right)(L^2+E^2H^2)\eta_{\nu(n)}^2.
\end{eqnarray*}  }
Since $\sum\limits_{n=1}^{\infty}\eta_{\nu(n)}^2 < \infty$, we obtain that, for all $i=1,2,\ldots,m$,
\begin{equation}\label{2-m-ic-5}
\lim\limits_{n\rightarrow\infty}\|x^{i,\nu(n)}-x^{i-1,\nu(n)}\|=0.
\end{equation}
 By utilizing this, along with the nonexpansivity of each $T_i$ and the fact that  $\eta_{\nu(n)}\to0$,
 we derive that
\begin{equation}\label{2-m-ic-52}
\lim\limits_{n\rightarrow\infty}\|T_ix^{i-1,\nu(n)}-x^{i-1,\nu(n)}\|=0,
\end{equation}
for all $i=1,2,\ldots,m$. 
Moreover, Lemma \ref{incre-lm-1}(i) and the relation (\ref{2-m-ic-eq7}) yield that 
\begin{eqnarray*}  2g(z)\eta_{\nu(n)}(F(x^{\nu(n)})-F(z))     &\le& 4M\left(N+\frac{1}{N}\right)(L^2+E^2H^2)\eta_{\nu(n)}^2.
\end{eqnarray*}
Since $\{\eta_{\nu(n)}\}_{n=1}^\infty\subset(0,\infty)$ and 
$\eta_{\nu(n)}\to0$, we deduce that
\begin{equation}\label{2-eq10}
\limsup\limits_{n\rightarrow\infty}F(x^{\nu(n)})\le F(z).
\end{equation}
Choose a subsequence  $\{x^{\nu(n_k)}\}_{k=1}^\infty$ of $\{x^{\nu(n)}\}_{n=1}^\infty$ arbitrarily. From the above, we obtain 
\begin{equation}\label{2-eq11}
\limsup\limits_{k\rightarrow\infty}F(x^{\nu(n_k)})\le \limsup\limits_{n\rightarrow\infty}F(x^{\nu(n)})\le F(z).
\end{equation}
Furthermore, since $\{x^{\nu(n_k)}\}_{k=1}^\infty$ is bounded, there is a subsequence $\{x^{\nu(n_{k_j})}\}_{j=1}^\infty$ of $\{x^{\nu(n_k)}\}_{k=1}^\infty$ such that $x^{\nu(n_{k_j})}\to x^*\in\R^k.$ 
By following similar arguments to those used in proving that $x^*\in\bigcap_{i=1}^m\fix T_i$ in \textbf{Case 1} along with the inequalities (\ref{2-m-ic-5}) and (\ref{2-m-ic-52}), we obtain that
$x^{\nu(n_{k_j})}\to x^*\in\bigcap_{i=1}^m\fix T_i.$ 
Then, the lower semicontinuity of $F$, the boundedness of $\{F(x^{\nu(n_{k_j})})\}_{j=1}^\infty$ and the inequality (\ref{2-eq11})  lead to 
\begin{equation*}
F(x^*)\le\liminf\limits_{j\rightarrow\infty} F(x^{\nu(n_{k_j})}) \le \limsup\limits_{j\rightarrow\infty} F(x^{\nu(n_{k_j)}})\le F(z),
\end{equation*}
for all $z\in \bigcap_{i=1}^m\fix T_i$. Thus, we have $x^*\in \mathcal{Z}.$ Now, since $x^{\nu(n_{k_j})}\to x^*$, in view of (\ref{2-m-ic-eq8}), we have
\begin{equation*}
0\le\limsup\limits_{j\rightarrow\infty}\|x^{n_{k_j}}-x^*\| \le\limsup\limits_{j\rightarrow\infty}\|x^{\nu(n_{k_j})}-x^*\|=0.
\end{equation*}
Hence, this concludes that $\lim\limits_{j\rightarrow\infty} x^{n_{k_j}}=x^*\in \mathcal{Z}$. Therefore, we obtain that there exists a subsequence of $\{x^n\}_{n=1}^\infty$ that converges to a point in $\mathcal{Z}$. This completes the proof.  
\end{proof}

\section{Numerical Examples}\label{sect-numerical}

In this section, we present three numerical examples to illustrate the performance of the proposed methods in solving various types of fractional programming problems. 
 All experiments were implemented in MATLAB R2025a and executed on a MacBook Pro (14-inch, 2021) with an Apple M1 Pro chip and 16 GB memory.

\subsection{Minimizing the Quadratic-to-Linear Ratio over the Underdetermined System of Equations}

In this example, we consider the following smooth quaratic-linear fractional programming problem:
\begin{eqnarray}\label{FP-ex1}
\mathrm{minimize} &&\frac{0.5\< x,Qx\>}{\<s,x\>}\nonumber \\
{  \mathrm{subject} \hspace{0.1cm}  \mathrm{ to}} && Ax= b, \\
              && x\in [10^{-8},10^8]^k \nonumber
\end{eqnarray}
where $Q$ is an $k\times k$ positive definite symmetric matrix defined by $Q:=P^\top P+kI_n,$ where the $k\times k$ matrix $P$ is randomly generated in $(0,1)$, and $I_n$ is the identity $k\times k$ matrix. The entries of vectors $s$ and $b$ are randomly generated in $(0,1)$.
The matrix $A$ is an $m\times k$ matrix, with $m<k$, where the entries are randomly generated in $(0,1)$.
Note that the matrix $A$ is the $m\times k$ matrix with $m<k$, we have the system $Ax=b$ is the underdetermined system and so such system has infinitely many solutions.

In this experiment, we compare the performance of the FSSM with Dinkelbach's work, the DA \cite{D67}, and the PGA proposed by Bo\c{t} and Csetnek \cite{BotC17} in solving the problem (\ref{FP-ex1}).
For the FSSM, we consider both constant (FSSM-C) and diminishing (FSSM-D) step sizes. We set the possibly fine-tuned selected step sizes for FSSM-C and FSSM-D as \( \eta_n = 5.1\times10^{-5}\) and \( \eta_n = \frac{8.9\times10^{-5}}{n+1} \), respectively. Since the DA and the PGA require solving subproblems, we use the HSDM proposed by Yamada \cite{Y01} to address them for each iteration $n$.
We define the main operator $T$ used in the FSSM and the HSDM as $T:=P_{C_2}P_{C_1}$, where $C_1:=\{x\in\R^k: Ax=b\},$ and  $C_{2}:=\{x\in\R^k:x\in [10^{-8},10^8]^k\}$ is the box constraint.  
Note that, for the HSDM, the parameter \( \alpha_j \) for the DA and the PGA were empirically selected as \( \alpha_j = \frac{3\times 10^{-6}}{1 + \theta_n}\) and  \( \alpha_j = \frac{10^{-4}}{1 + 4\|A\|\theta_n}\), respectively. We set the number of inner-loop iterations for the DA and the PGA is fixed at $10$. We put the initial vector for these three tested methods as the $k\times1$ vector whose the all components is $0.1$. We performed the experiment using $100$ samples for each randomly generated matrices and vectors with $k=1000$ and $m=50$, and the results given in Figure \ref{fig:ql} were averaged across these trials.
	\begin{figure}[H]
			\begin{center}
				\subfigure{
					\resizebox*{5.2cm}{!}{\includegraphics{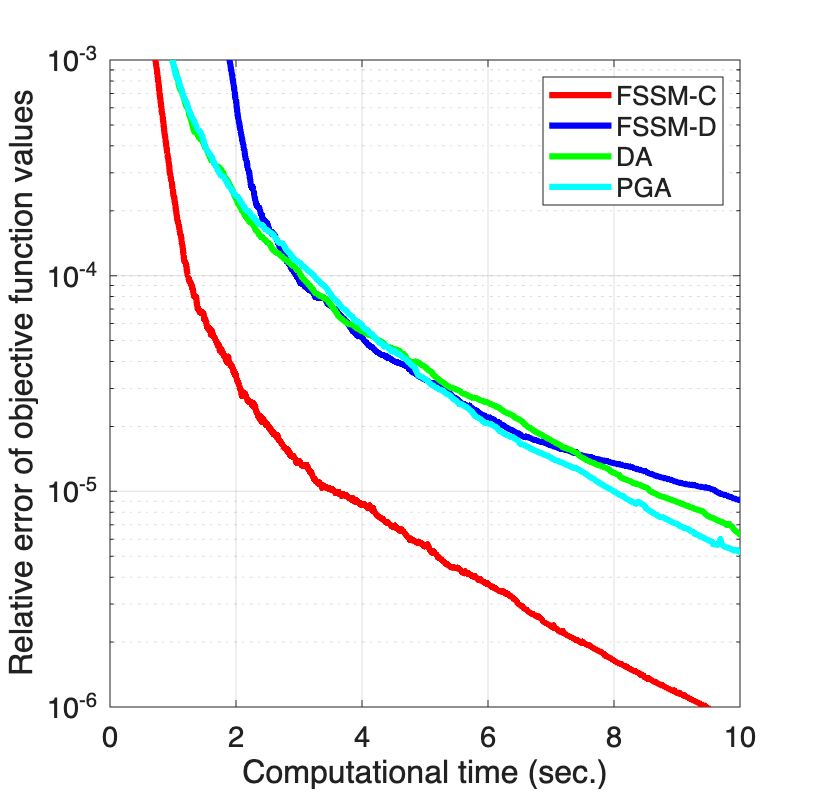}}}
                    \subfigure{
					\resizebox*{5.2cm}{!}{\includegraphics{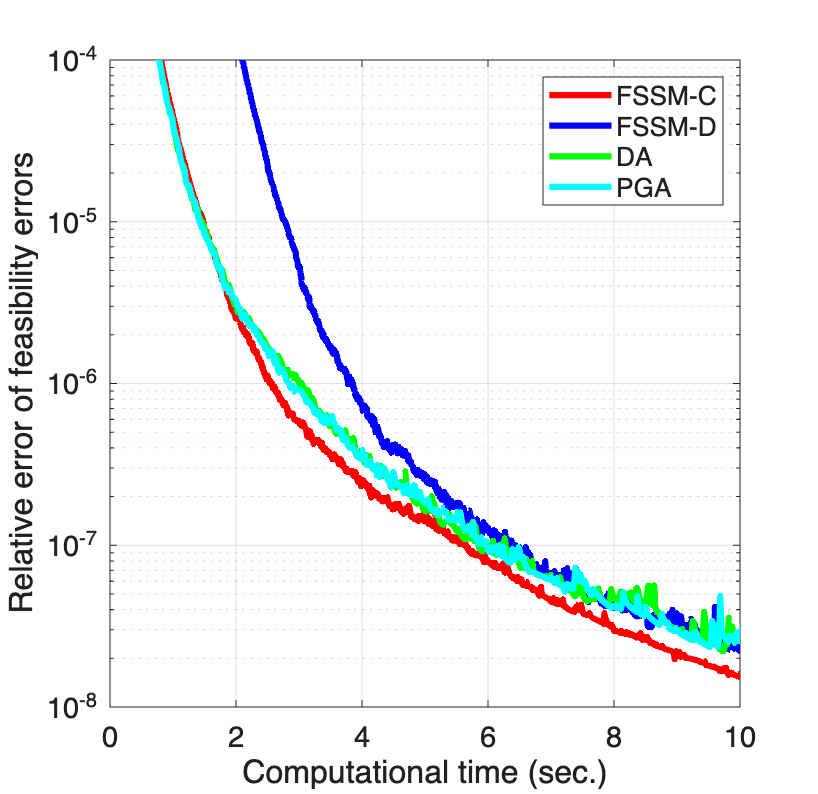}}}
				\subfigure{
					\resizebox*{5.2cm}{!}{\includegraphics{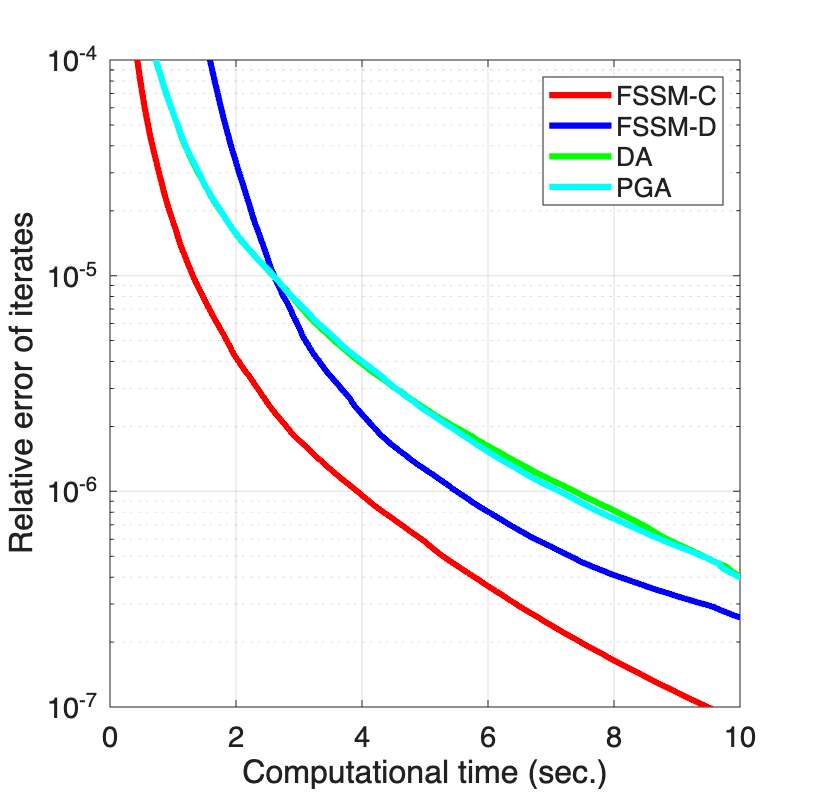}}}
				\caption{Results from  FSSM, DA \cite{D67} and PGA \cite{BotC17}                
                within 10 seconds}
				\label{fig:ql}
			\end{center}
		\end{figure}

From Figure \ref{fig:ql}, we illustrate the performance in terms of the relative errors of objective function values, feasibility errors, and iterates. The relative error of objective function values is defined as 
\(
\frac{\left|\theta_{n+1} - \theta_n\right|}{\theta_n +1},
\)
the relative error of feasibility errors is defined as 
\(
\frac{\|Ax^{n+1} - b\| - \|Ax^{n} - b\|}{\|Ax^{n} - b\|+1}
\), and the relative error of  iterates is defined as 
\(
\frac{\|x^{n+1} - x^{n}\|}{\|x^{n}\|+1}
\).
As observed, under the experimental setup, the FSSM-C demonstrates the best performance compared to the FSSM-D, the DA, and the PGA by achieving the least relative errors of all types. Moreover, both the FSSM-C and FSSM-D achieve lower relative errors in the iterates than the DA and the PGA. This highlights the advantage of the proposed method.

\subsection{Minimizing the Cost-to-Profit Ratio over the Constraints on Funding Levels}

In this example, we consider a variation of the profit-to-cost maximization model, which incorporates a Cobb–Douglas production function in the denominator. This model was originally introduced by Bradley and Frey \cite{BF74}, and was later considered by Hu et al. \cite{HYS15} and Hishinuma and Iiduka \cite{HI20}. The objective of the original model is to maximize production efficiency, defined as the ratio of profit to cost, subject to funding-level constraints. 
In our experiment, rather than maximizing production efficiency, we minimize the cost-to-profit ratio, which measures the amount of expenditure required to generate one unit of profit. Minimizing this ratio corresponds to improving cost-efficiency and is mathematically equivalent to maximizing production efficiency from an inverse perspective, provided the ratio remains positive. This formulation aligns with practical concerns such as budget constraints, cost-reduction goals, and sustainable resource management, which organizations may face when aiming to optimize resource use while maintaining profitability.

Let \( l = 1,2\ldots, p \) denote a set of projects, and let \( j = 1,2, \ldots, k \) denote a set of production factors. The variables $x^j$ represent the production factors. We consider the following problem: 
\begin{equation} \label{1ex1-P-Incremental}\displaystyle
 \tag{CPP}
	\begin{array}{ll}
\mathrm{minimize}\hskip0.2cm\indent  \displaystyle \frac{\sum_{j=1}^k c_{j}x^j + c_0}{a_0\prod_{j=1}^k (x^{j})^{a_j}}\\
	\mathrm{subject \hskip0.2cm to}\indent \displaystyle  \underline{q}_l \leq \sum_{j=1}^k b_{lj}x^j \leq  \overline{q}_l, \indent l=1,\ldots,p,\\ 
    \indent\hspace{1.8cm} 
    x^{j}\in [10^{-8},10^8], \indent j=1,\ldots,k,
	\end{array}  
 \end{equation}
where, for \( l = 1, \ldots, p \) and \( j = 1, \ldots, k \), the parameters are randomly generated as follows:
\( c_j \) and \( a_j \) are randomly generated in the interval \((0, k)\) such that $\sum_{j=1}^ka_j=1$; \( c_0 \) and \( a_0 \) are in \((1, 10)\); \( b_{lj} \) are in \((0, 1)\); and the lower and upper bounds \(\underline{q}_l\) and \(\overline{q}_l\) lie in \(\left(0, 25\|b_l\|\right)\) and \(\left(75\|b_l\|, 100\|b_l\|\right)\), respectively, where \( b_l := (b_{l1}, b_{l2}, \ldots, b_{lk}) \).

According to Problem (\ref{1ex1-P-Incremental}), the numerator of the objective function represents the total cost, formulated as a linear (and therefore convex) function of the investment levels in the various projects.
The denominator is a Cobb–Douglas production function representing the total profit generated across all projects, and it is concave \cite[Theorem 2.4.1]{CM09}. 
Since the objective is to minimize the ratio of a positive convex function to a positive concave function, this problem fits within the framework of Problem (\ref{main-FP}).

 In what follows, we apply the proposed FSSM to solve Problem (\ref{1ex1-P-Incremental}) and compare its performance with the {\it fixed point quasiconvex subgradient method} (FPQSM) \cite[Algorithm 1]{HI20}. 
In this experiment, we consider four variants of the FSSM, differing in the operator \( T \) (cyclic or simultaneous) and the step size \( \eta_n \) (constant or diminishing): FSSM-C-C (cyclic \( T \), constant \( \eta_n \)), FSSM-C-D (cyclic \( T \), diminishing \( \eta_n \)), FSSM-S-C (simultaneous \( T \), constant \( \eta_n \)), and FSSM-S-D (simultaneous \( T \), diminishing \( \eta_n \)). We also consider two variants of the FPQSM based on different step sizes: FPQSM-C (constant \( \eta_n \)) and FPQSM-D (diminishing \( \eta_n \)).

Let $x := (x_j)_{j=1}^k \in \mathbb{R}^k$. The operator $T$ is defined as follows: for the FSSM-C variants (cyclic update), we use $T := P_{C_{2p+1}} P_{C_{2p}} \cdots P_{C_1}$; and for the FSSM-S variants (simultaneous update) and the FPQSM, we use $T := P_{C_{2p+1}} \left( \frac{1}{2p} \sum_{l=1}^{2p} P_{C_l} \right)$, 
where each \( C_l \subset \mathbb{R}^k \) is a closed convex set. Specifically, for \( l = 1, \ldots, p \), we define 
\( C_l := \{x \in \mathbb{R}^k : \langle -b_l, x \rangle \leq -\underline{q}_l\} \), 
and for \( l = p+1, \ldots, 2p \), 
\( C_l := \{x \in \mathbb{R}^k : \langle b_{l - p}, x \rangle \leq \overline{q}_{l - p}\} \). 
The final set is the box constraint, \( C_{2p+1} := \{x \in \mathbb{R}^k : x \in [10^{-8},10^8]^k\} \).

Firstly, we consider the behavior and performance of the FSSM and the FPQSM when $k=p=500$.
We chose $\eta_n = \frac{0.1}{k}$ for the FSSM-C-C, FSSM-S-C and FPQSM-C, and $\eta_n = \frac{0.1}{n+1}$ for the FSSM-C-D, FSSM-S-D and FPQSM-D, with \(\alpha_n = 0.5\) used for both variants of the FPQSM.  We initialized the vector with $k$ components, all set to $1$, i.e., a 
$k\times1$ vector of ones. We performed 10 independent tests and the averaged results are shown in Figure \ref{ex2-plot-result},
 in which the relative error of objective function values 
and the relative error of the iterates are defined as the above example,  
and the relative error of feasibility errors is defined as 
\( 
\frac{|\mathrm{FE}(x^{n+1})| - |\mathrm{FE}(x^{n})|}{\mathrm{FE}(x^{n+1})+1}
\), 
where the feasibility error is defined by $\mathrm{FE}(x^{n}):=\frac{1}{2p}\left(\sum_{l=1}^p\max\{\<-b_l,x^n\> +\underline{q}_l,0\} + \sum_{l=1}^p\max\{\<b_{l},x^n\> - \overline{q}_{l},0\}\right).$

\begin{figure}[H]
			\begin{center}
				\subfigure{
					\resizebox*{5.2cm}{!}{\includegraphics{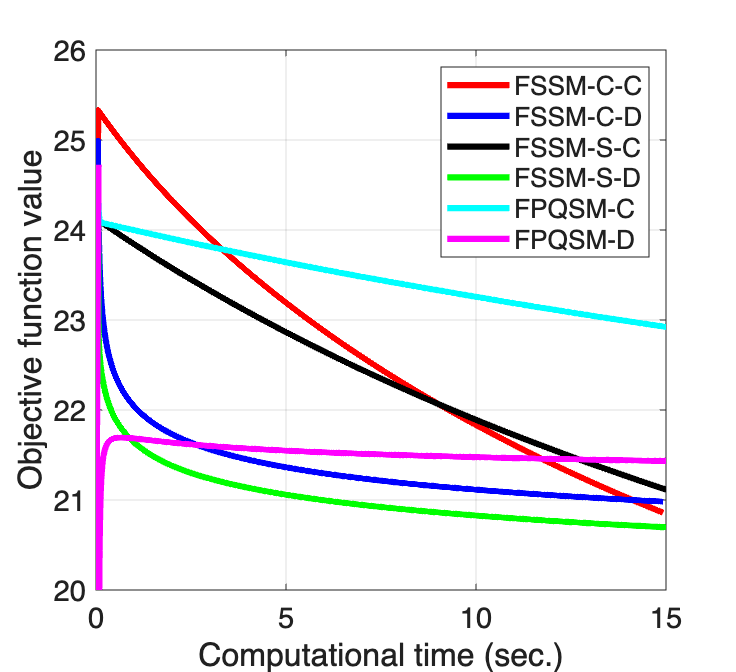}}}
				\subfigure{
					\resizebox*{5.2cm}{!}{\includegraphics{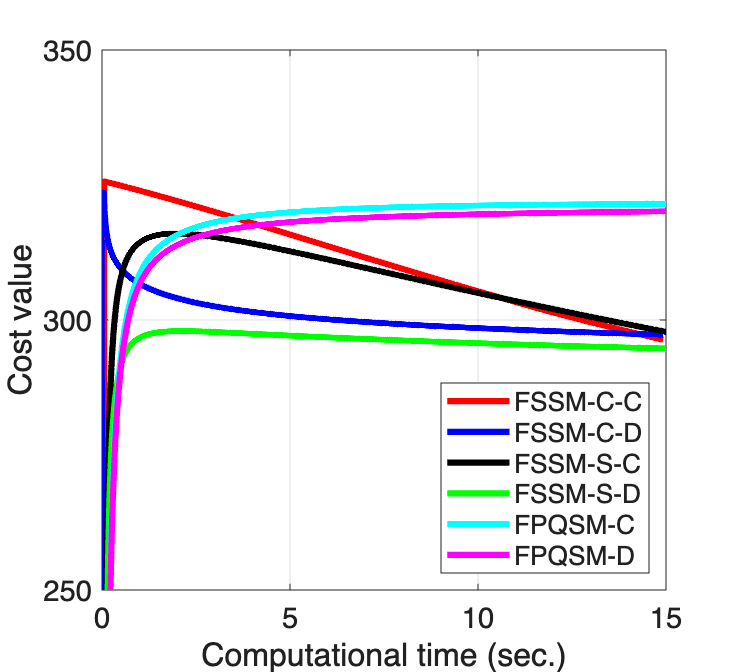}}}
                    \subfigure{
					\resizebox*{5.2cm}{!}{\includegraphics{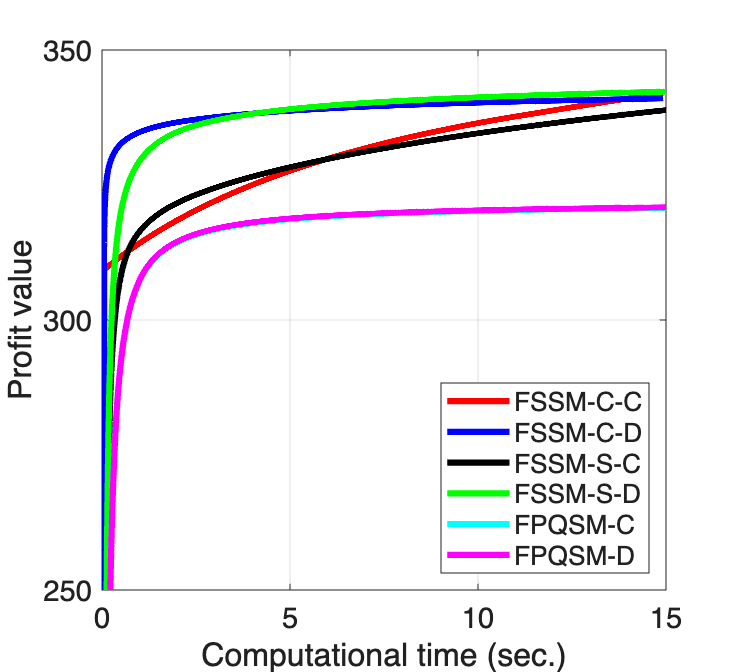}}}
                    \subfigure{
					\resizebox*{5.2cm}{!}{\includegraphics{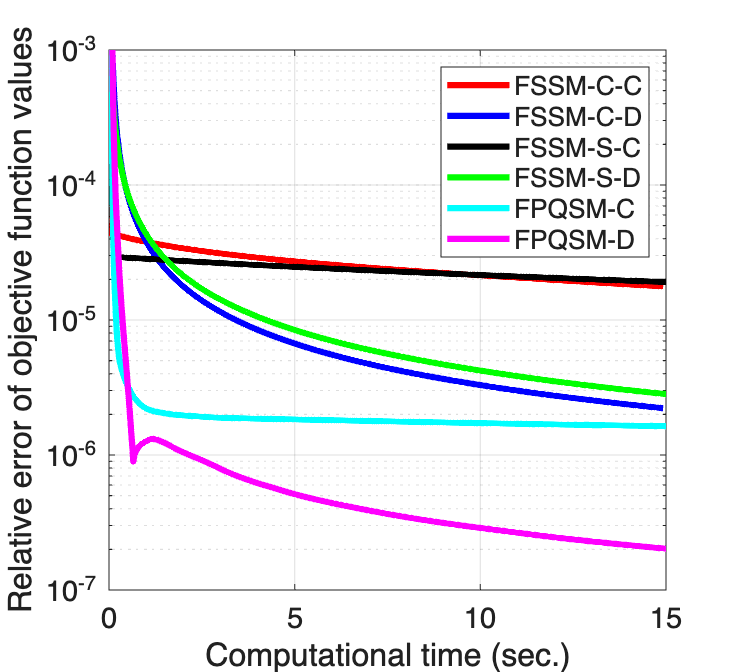}}}
				\subfigure{
					\resizebox*{5.2cm}{!}{\includegraphics{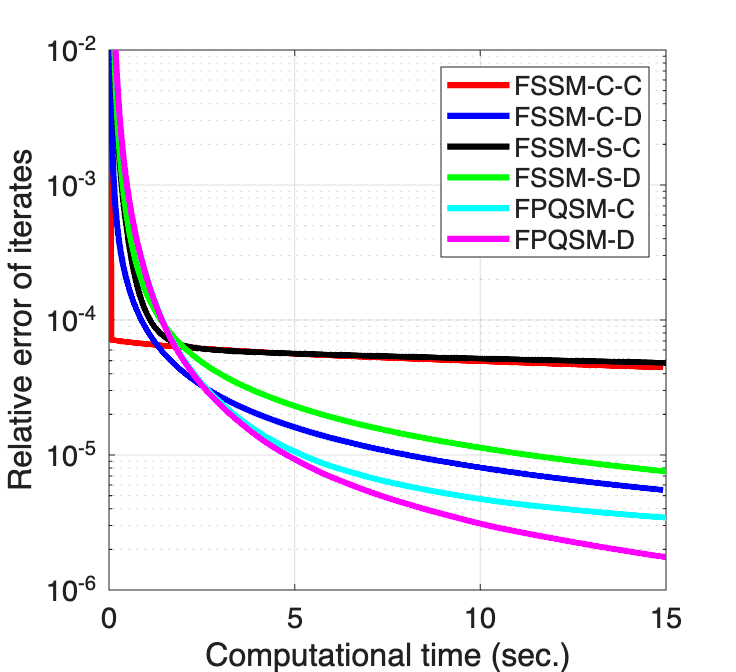}}}
                    \subfigure{
					\resizebox*{5.2cm}{!}{\includegraphics{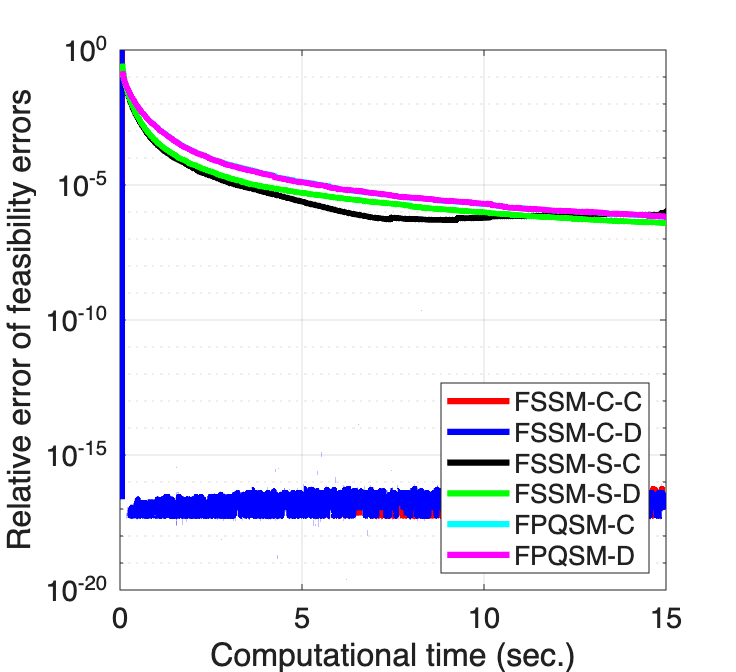}}}
				\caption{Results for various cases of FSSM and FPQSM}
				\label{ex2-plot-result}
			\end{center}
		\end{figure}

According to Figure \ref{ex2-plot-result}, under the experimental setup, all four variants of our algorithm successfully reduce costs and increase profits while maintaining low objective function values, aligning well with the optimization goals. In particular, all four variants of our method achieve lower costs and higher profits, and yield the lowest objective function values compared to the FPQSM.

Next, we examine the behavior and performance of all variants of our method and of the FPQSM across varying numbers of $k$ and $p$. We chose \(\eta_n = \frac{0.1}{k}\) for the FSSM-C-C, FSSM-S-C, and FPQSM-C, and \(\eta_n = \frac{0.5}{n+1}\) for the FSSM-C-D, FSSM-S-D, and FPQSM-D, with \(\alpha_n = 0.5\) used for both variants of the FPQSM.
We initialized the vector with \(k\) components, all set to 1, i.e., a \(k \times 1\) vector of ones. We performed 100 independent tests, each running for 10,000 iterations. 

The results, averaged over these trials, are presented in Table \ref{ex2-tab:comparison}, in which {\it Obj.} refers to the objective function values, {\it Rel. Obj.} refers to the relative error of objective function values, {\it Rel. Iter.} refers to the
relative error of iterates, and {\it Rel. Fea.} refers to the relative error of
feasibility errors. All of these quantities are defined as above.

\begin{table}[H]
{\small	\centering 
    \begin{tabular}{l l l r  r r  r r  r r r  r }	\toprule				
$k$ 	&	$p$	&	Method	&	Obj.     	&	Cost	&	Profit    	&	Rel. Obj.      	&	Rel. Iter.            	&	Rel. Fea.       	&	Time	\\  \midrule 
10	&	10	&	FSSM-C-C         	&	3.50	&	28.17	&	40.44	&	1.59e-06	&	2.07e-05	&	8.97e-17	&	0.60	\\
	&		&	FSSM-C-D         	&	4.65	&	35.54	&	39.61	&	7.11e-07	&	1.75e-06	&	8.88e-17	&	0.60	\\
	&		&	FSSM-S-C         	&	3.51	&	28.02	&	40.15	&	1.59e-06	&	2.09e-05	&	9.98e-08	&	0.61	\\
	&		&	FSSM-S-D         	&	4.70	&	33.82	&	37.79	&	6.92e-07	&	1.80e-06	&	2.02e-09	&	0.61	\\
	&		&	FPQSM-C          	&	3.88	&	34.73	&	35.67	&	3.90e-06	&	9.09e-06	&	6.20e-09	&	0.64	\\
	&		&	FPQSM-D          	&	4.59	&	35.99	&	34.54	&	5.49e-08	&	8.30e-08	&	2.51e-10	&	0.64	\\  \midrule 
20	&	20	&	FSSM-C-C         	&	2.15	&	44.69	&	65.98	&	1.98e-06	&	2.33e-05	&	1.12e-16	&	1.05	\\
	&		&	FSSM-C-D         	&	2.78	&	56.12	&	64.96	&	8.08e-07	&	2.18e-06	&	6.39e-17	&	1.08	\\
	&		&	FSSM-S-C         	&	2.16	&	44.44	&	65.56	&	2.01e-06	&	2.37e-05	&	1.65e-07	&	1.09	\\
	&		&	FSSM-S-D         	&	2.79	&	54.03	&	62.98	&	7.98e-07	&	2.24e-06	&	5.08e-09	&	1.07	\\
	&		&	FPQSM-C          	&	2.70	&	57.58	&	57.46	&	2.76e-06	&	5.19e-06	&	3.50e-09	&	1.10	\\
	&		&	FPQSM-D          	&	2.83	&	58.26	&	56.83	&	3.86e-08	&	6.24e-08	&	9.01e-11	&	1.12	\\  \midrule 
50	&	50	&	FSSM-C-C         	&	7.05	&	71.80	&	100.00	&	2.82e-06	&	2.47e-05	&	1.01e-16	&	2.50	\\
	&		&	FSSM-C-D         	&	8.29	&	86.86	&	103.24	&	9.31e-07	&	2.81e-06	&	1.11e-16	&	2.51	\\
	&		&	FSSM-S-C         	&	7.06	&	71.15	&	99.14	&	2.74e-06	&	2.45e-05	&	2.27e-07	&	2.57	\\
	&		&	FSSM-S-D         	&	8.20	&	84.24	&	101.58	&	8.76e-07	&	2.79e-06	&	1.43e-08	&	2.55	\\
	&		&	FPQSM-C          	&	7.39	&	95.22	&	92.41	&	1.33e-06	&	2.69e-06	&	1.04e-08	&	2.57	\\
	&		&	FPQSM-D          	&	9.89	&	97.67	&	92.29	&	6.64e-08	&	9.44e-08	&	1.00e-09	&	2.60	\\  \midrule 
80	&	80	&	FSSM-C-C         	&	2.52	&	92.68	&	133.29	&	3.51e-06	&	2.58e-05	&	4.17e-17	&	3.98	\\
	&		&	FSSM-C-D         	&	2.80	&	107.37	&	136.78	&	8.82e-07	&	3.02e-06	&	4.00e-17	&	4.07	\\
	&		&	FSSM-S-C         	&	2.51	&	91.60	&	131.98	&	3.28e-06	&	2.51e-05	&	2.01e-07	&	4.13	\\
	&		&	FSSM-S-D         	&	2.78	&	104.81	&	135.32	&	8.04e-07	&	2.94e-06	&	2.36e-08	&	4.06	\\
	&		&	FPQSM-C          	&	3.12	&	123.91	&	122.10	&	7.67e-07	&	1.39e-06	&	3.34e-09	&	4.10	\\
	&		&	FPQSM-D          	&	3.19	&	123.92	&	122.14	&	3.75e-08	&	6.39e-08	&	2.72e-10	&	4.13	\\  \midrule 
100	&	100	&	FSSM-C-C         	&	2.31	&	105.54	&	135.92	&	3.93e-06	&	2.57e-05	&	3.62e-17	&	5.12	\\
	&		&	FSSM-C-D         	&	3.57	&	121.66	&	139.78	&	8.59e-07	&	3.01e-06	&	4.97e-17	&	5.03	\\
	&		&	FSSM-S-C         	&	2.30	&	104.16	&	134.41	&	3.63e-06	&	2.48e-05	&	2.11e-07	&	5.25	\\
	&		&	FSSM-S-D         	&	2.49	&	116.38	&	138.29	&	7.73e-07	&	2.93e-06	&	2.82e-08	&	5.20	\\
	&		&	FPQSM-C          	&	3.02	&	140.00	&	126.28	&	7.75e-07	&	1.29e-06	&	2.87e-09	&	5.29	\\
	&		&	FPQSM-D          	&	3.03	&	139.82	&	126.38	&	4.00e-08	&	6.60e-08	&	2.59e-10	&	5.36	\\ \bottomrule				\end{tabular}		 }								
    \caption{\label{compare--tb1} Results for  FSSM-C, FSSM-S, and FPQSM \cite{HI20} across different dimensions ($k$) and numbers of constraints ($p$).}       				    \label{ex2-tab:comparison}						\end{table}										
    
As shown in Table~\ref{ex2-tab:comparison}, under the experimental setup, across all values of \(k\) and \(p\), the FSSM-C-C and the FSSM-S-C outperform the FPQSM by achieving lower costs, higher profits, and lower objective function values. Meanwhile, the FSSM-C-C and the FSSM-C-D outperform the FPQSM in terms of relative feasibility errors. In addition, all variants of our algorithm require less computational time than the FPQSM, highlighting the overall advantage of our method.

\subsection{Minimizing a Finite Sum of Ratios of Linear Functions}

In this example, we consider the following sum of ratios optimization problem:
\begin{equation} \label{ex2-P-Incremental}\displaystyle
 \tag{FSP}
	\begin{array}{ll}
\mathrm{minimize}\hskip0.2cm\indent  \displaystyle F(x):=\sum\limits_{i=1}^{m}\frac{\sum_{j=1}^k c_{ij}x^j + r^i}{\sum_{j=1}^k d_{ij}x^j + s^i}\\
	\mathrm{subject \hskip0.2cm to}\indent \displaystyle  \sum_{j=1}^k a_{lj}x^j \leq  b_l, \indent l=1,\ldots,p,\\ 
    \indent\hspace{1.8cm} x^j\in[0,100], \indent j=1,\ldots,k.
	\end{array}  
 \end{equation}
where $c_{ij}, d_{ij}$, $i = 1, 2, \ldots, m$, $j = 1, 2,\ldots, k$, are all randomly generated in the interval $(0.1, 10)$, $r^{i}, s^{i}$, $i = 1, 2,\ldots, m$, are all randomly generated in the interval $(0,1)$, $a_{lj}$, $l = 1, 2, \ldots, p$, $j = 1, 2,\ldots, k$, are all randomly generated in the interval $(-10, 10)$ and $b_l$, $l =
1, 2,\ldots, p$, are all randomly generated in the interval $(0, 10)$.

In this example, we apply the proposed IFSSM and an existing method to solve Problem (\ref{ex2-P-Incremental}). 
For \( x := (x^j)_{j=1}^k \in \mathbb{R}^k \), we define the main operator used in the IFSSM as 
$T_l:= P_{C_l}$,
where each $C_l\subset\R^k$ is defined as follows. For  $l=1,2,\ldots,p$, the constraint set is given by  $C_l:=\{x\in\R^k:\sum_{j=1}^k a_{lj}x_j \leq  b_l \},$ 
and for $l=p+1$, the constraint set is  
$C_{p+1}:=[0,100]^k$.
Note that the IFSSM is proposed for the setting where the number of objective functions and the number of constraints are equal. In this experiment, for cases where these numbers differ, we apply the following strategy:
    if $m<p+1$, 
    we set $f_{i}:=0$ for all $i=m+1, m+2,\ldots,p+1$;
     if $m>p+1$, 
     we set $T_{i}:=Id$ for all $i=p+2, p+3, \ldots,m$.
 We initialized the vector with $k$ components, all set to $1$, i.e., a 
$k\times1$ vector of ones. We set $\eta_n=\frac{1}{n+1}$.
The experimental methods were terminated using the following stopping criterion:
$$\max\left( \frac{\|x^{n+1} - x^n\|}{\|x^n\| + 1}, \frac{|F(x^{n+1}) - F(x^n)|}{F(x^n) +1} \right) \le 10^{-5}.$$

Since the ratio of linear functions is quasi-convex, we employ the {\it incremental quasi-subgradient method} (IPQSM), proposed in \cite[Algorithm 1]{HYY19}, as a benchmark. 

Note that at each subiteration, the IPQSM utilizes the metric projection as its main operator. 
However, due to the complex structure of the constraints in Problem (\ref{ex2-P-Incremental}), computing metric projections onto such sets is difficult.
To address this issue, we apply the classical {\it Halpern iteration} as an inner approximation scheme.
Specifically, for each $i=1,\ldots,m$ and $s\in\N$, 
we compute
$$u^{i,s+1} := \lambda_s(x^{i-1,n} - v_n \nabla\tilde{f}_i(x^{i-1,n})) + (1-\lambda_s)T(u^{i,s}),$$
 where $T:=P_{C_{p+1}}P_{C_p}P_{C_{p-1}}\cdots P_{C_1}$. 
Note that under certain conditions on $\{\lambda_s\}_{s=1}^\infty$, the sequence $\{u^{i,s}\}_{s=1}^\infty$ converges to the projection $P_C(x^{i-1,n} - v_n \nabla\tilde{f}_i(x^{i-1,n}) )$, where $C=\bigcap_{l=1}^{p+1}C_l$; see \cite[Theorem 30.1]{BC17} for further details.
For the benchmark approach, we set the initial point $x^1$ and the initial inner point $u^1$ as  vectors with all coordinates equal to $1$, and set $v_n=\lambda_n = \frac{1}{n+1}$ for all $n \in \mathbb{N}$. The number of inner-loop iterations for the IPQSM is fixed at 10. 
We conduct 10 independent random trials. The average number of iterations ($\#\mathrm{Iters}$), computational runtimes, and final objective values (Func. Val.) are reported in Table \ref{tab1}. 

\begin{table}[H]
	
	\caption{Results for IFSSM and IPQSM \cite{HYY19} across different dimensions ($k$), numbers of ratio terms in the objective function ($m$), and numbers of linear inequality constraints ($p$).}\label{tab1}
	\centering
	
	\begin{tabular}{ l l l r r r r r r r r r r r}\toprule
		\multirow{2}{*}{$k$}&\multirow{2}{*}{$m$}&\multirow{2}{*}{$p$}
		& \multicolumn{3}{c}{IFSSM} &\multicolumn{3}{c}{IPQSM \cite{HYY19}} \\ 
        \cmidrule(l){4-6}  \cmidrule(l){7-9} 
		& & & $\#$Iters &Time& Func. Val. &$\#$Iters &Time & Func. Val.\\ \midrule
5	&	5	&	10	&	5331	&	0.41	&	3.83	&	12028	&	17.75	&	5.76	\\
	&		&	50	&	1323	&	0.20	&	5.14	&	1398	&	8.43	&	5.08	\\
	&		&	100	&	2258	&	0.62	&	4.52	&	1570	&	18.70	&	4.22	\\
	&	10	&	10	&	5013	&	0.49	&	12.48	&	6495	&	19.84	&	10.77	\\
	&		&	50	&	1606	&	0.30	&	11.65	&	3015	&	39.94	&	12.54	\\
	&		&	100	&	1743	&	0.51	&	11.66	&	1186	&	27.90	&	13.45	\\
	&	50	&	10	&	2440	&	0.37	&	117.01	&	7017	&	94.98	&	157.49	\\
	&		&	50	&	2770	&	1.04	&	55.05	&	2597	&	151.84	&	60.00	\\
	&		&	100	&	2033	&	0.99	&	60.49	&	2316	&	265.28	&	71.42	\\
	&	100	&	10	&	2558	&	0.60	&	126.97	&	10407	&	282.16	&	217.17	\\
	&		&	50	&	2854	&	1.30	&	111.69	&	2654	&	310.71	&	146.87	\\
	&		&	100	&	2465	&	1.79	&	114.35	&	1846	&	423.77	&	144.62	\\  \midrule 
10	&	5	&	10	&	8075	&	0.86	&	5.44	&	12326	&	17.87	&	5.31	\\
	&		&	50	&	2093	&	0.34	&	6.23	&	6671	&	39.55	&	6.62	\\
	&		&	100	&	1525	&	0.42	&	5.04	&	2005	&	23.07	&	5.02	\\
	&	10	&	10	&	9348	&	1.29	&	10.16	&	10055	&	28.00	&	11.20	\\
	&		&	50	&	3168	&	0.61	&	11.20	&	6063	&	71.27	&	15.89	\\
	&		&	100	&	4270	&	1.34	&	10.05	&	2567	&	59.02	&	14.75	\\
	&	50	&	10	&	4262	&	0.71	&	71.29	&	9036	&	122.86	&	73.83	\\
	&		&	50	&	2849	&	1.08	&	53.97	&	5809	&	340.78	&	58.41	\\
	&		&	100	&	2781	&	1.38	&	53.32	&	2165	&	248.59	&	57.20	\\
	&	100	&	10	&	3723	&	0.91	&	122.04	&	9454	&	256.70	&	234.35	\\
	&		&	50	&	3184	&	1.48	&	108.28	&	4281	&	502.21	&	149.73	\\
	&		&	100	&	3156	&	2.33	&	103.89	&	4144	&	951.79	&	129.09	\\  \midrule 
50	&	5	&	10	&	6926	&	1.65	&	4.55	&	7027	&	11.88	&	3.95	\\
	&		&	50	&	12254	&	5.89	&	5.37	&	3629	&	23.77	&	5.15	\\
	&		&	100	&	6921	&	3.26	&	5.65	&	6946	&	89.26	&	5.67	\\
	&	10	&	10	&	5740	&	1.34	&	10.43	&	7719	&	24.67	&	11.03	\\
	&		&	50	&	9345	&	4.03	&	10.34	&	6220	&	81.49	&	11.33	\\
	&		&	100	&	9261	&	5.17	&	9.45	&	7528	&	192.89	&	9.42			 \\ \bottomrule
	\end{tabular}
	
\end{table}

As demonstrated in Table \ref{tab1}, under the experimental setup, the IFSSM consistently outperforms the IPQSM \cite{HYY19} by requiring less computational time across all tested scenarios. This advantage becomes even more marked for larger problem sizes, which may be attributed to the fact that proposed method does not require solving subproblems.
\section{Concluding Remarks}

This paper addressed a class of nonsmooth fractional programming problems with fixed-point constraints, where the numerator is convex and the denominator is concave. To solve this problem, we proposed the FSSM,
which computes the subgradients for the convex numerator and concave denominator separately, updating each function independently during the optimization process. We established the subsequential convergence of the FSSM under specific conditions on the step size and the boundedness of the sequence $\{x_n\}_{n=1}^{\infty}$.
To eliminate the boundedness assumption of the FSSM, we introduced a modified version of the algorithm which is the AFSSM. We then proved its strong convergence under certain conditions related to the step size and the strong convexity of the numerator $f$.
Additionally, we proposed the IFSSM to address large-scale optimization problems. This method is designed for solving nonsmooth sum-of-ratios fractional programming problems over the intersection of fixed-point sets of firmly nonexpansive operators. We established the subsequential convergence for the IFSSM under appropriate assumptions.
The proposed methods offer a significant advantage by eliminating the need to solve subproblems at each iteration. This simplification reduces the overall complexity of the optimization process by removing the necessity of selecting specific subproblem-solving techniques or fine-tuning parameters, thereby potentially reducing both computational time and complexity in certain scenarios. Furthermore, by incorporating fixed-point constraints, the proposed methods are particularly well-suited for problems with complex constraint structures.
Finally, we evaluated the numerical performance of the proposed methods through three examples. The results demonstrate that our methods outperformed existing methods across a range of performance metrics within the given experimental setup.

Building on the algorithms developed in this paper, several future directions can be explored. 
One key direction is to generalize the methods to solve 
nonsmooth convex-convex fractional programming problems, and potentially extend them to non-convex optimization settings. The methods presented here establish subsequential convergence for two of the three proposed approaches, which, to the best of our knowledge, represent the furthest results achievable under the given assumptions. Further research could explore whether these results can be extended to strong convergence or other forms of convergence under more general conditions. Additionally, one might explore the possibility of removing the bounded assumption of the sequence $\{x^n\}_{n=1}^{\infty}$ in the FSSM, which could lead to further improvements in convergence.



\section*{Data availability}
The datasets generated and/or analysed during the current study are available from the corresponding author on reasonable request.

		\section*{Acknowledgement}
  M. Prangprakhon was partially
supported by Science Achievement Scholarship of Thailand (SAST), and Faculty of Science, Khon Kaen
University.


\section*{Funding}
This research was supported by the Fundamental Fund of Khon Kaen University. This research has received funding support from the National Science, Research and Innovation Fund or NSRF.

\section*{Conflict of interest}
The authors have no competing interests to declare that are relevant to the content of this article.

\end{document}